\documentclass[10pt,a4paper]{amsart} 
\usepackage{latexsym} 
\usepackage{amssymb}
\usepackage{enumerate} 

\usepackage{ifpdf} 
 
\usepackage{verbatim}
\usepackage{mathrsfs} 
\usepackage{multirow}
\ifpdf
 \usepackage[pdftex]{graphicx}
 \DeclareGraphicsExtensions{.pdf}
 \usepackage{hyperref}
\else
 \usepackage{graphicx}
 \DeclareGraphicsExtensions{.eps}
 \usepackage{hyperref}
\fi

\newcommand\rightmap[1]{\smash{\mathop{\rightarrow}\limits^{#1}}}

\newcommand{\cC}{\mathcal C}

\newcommand{\cM}{\mathcal M}

\newcommand{\cR}{\mathcal R}

\newcommand{\gs}{\sigma}
\newcommand{\e}{\varepsilon}

\newcommand{\Alb}{\text{\rm Alb}}
\newcommand{\Mor}{\text{\rm Mor}}

\newcommand{\MW}{\text{\rm MW}}
\newcommand{\Gr}{\text{\rm Gr}}
\newcommand{\Hom}{\text{\rm Hom}}
\newcommand{\End}{\text{\rm End}}
\newcommand{\Sing}{\text{\rm Sing}}

\newcommand{\AAA}{\mathbb A}
\newcommand{\ZZ}{\mathbb Z}
\newcommand{\CC}{\mathbb C}

\newcommand{\SSS}{\mathbb S}
\newcommand{\DD}{\mathbb D}
\newcommand{\QQ}{\mathbb Q}
\newcommand{\PP}{\mathbb P}
\newcommand{\NN}{\mathbb N}

\def\rk {{\rm rk\ }}

\def\Char {{\rm Char}}

\def\dim{{\rm dim}}

\def\deg{{\rm \,deg\, }}
\def\lcm{{\rm lcm}}

\def\Prod {{\displaystyle\prod}}

\def\cal{\mathcal}

\newtheorem{theorem}{Theorem}[section]
\newtheorem{lemma}[theorem]{Lemma}
\newtheorem{prop}[theorem]{Proposition}
\newtheorem{cor}[theorem]{Corollary}
\theoremstyle{definition}
\newtheorem{dfn}[theorem]{Definition}
 
\newtheorem{question}[theorem]{Question} 
\newtheorem{example}[theorem]{Example}

\theoremstyle{remark}
\newtheorem{remark}[theorem]{Remark}

\begin{document}

\title[Mordell-Weil groups of elliptic threefolds...]
{Mordell-Weil groups of elliptic threefolds 
and the Alexander module of plane curves}
\author{J.I.~Cogolludo-Agust{\'\i}n and A.~Libgober}
\address{Departamento de Matem\'aticas, IUMA\\
Universidad de Zaragoza\\
C.~Pedro Cerbuna 12\\
50009 Zaragoza, Spain}
\email{jicogo@unizar.es}

\address{
Department of Mathematics\\
University of Illinois\\
851 S.~Morgan Str.\\
Chicago, IL 60607}
\email{libgober@math.uic.edu}

\thanks{The first author was partially supported by the Spanish Ministry of 
Education MTM2010-21740-C02-02.
The second author was partially supported by NSF grant.}

\date{August 12, 2010}


\begin{abstract} 
We show that the degree of the Alexander polynomial of an irreducible plane algebraic 
curve with nodes and cusps as the only singularities does not exceed ${5 \over 3}d-2$ 
where $d$ is the degree of the curve. 
We also show that the Alexander polynomial $\Delta_\cC(t)$ of an irreducible curve 
$\cC=\{F=0\}\subset \PP^2$ whose singularities are nodes and cusps is non-trivial 
if and only if there exist homogeneous polynomials $f$, $g$, and $h$ such that $f^3+g^2+Fh^6=0$. 
This is obtained as a consequence of the correspondence, described here, 
between Alexander polynomials and ranks of Mordell-Weil groups of certain 
threefolds over function fields.
All results also are extended to the case of reducible curves and 
Alexander polynomials $\Delta_{\cC,\e}(t)$ corresponding to surjections 
$\e: \pi_1(\PP^2\setminus \cC_0 \cup \cC) \rightarrow \ZZ$, where $\cC_0$ is a line at infinity.
In addition, we provide a detailed description of the collection of 
relations of $F$ as above in terms of the multiplicities of the roots of 
$\Delta_{\cC,\e}(t)$. This generalization is made in the context of 
a larger class of singularities i.e. those which lead to
rational orbifolds of elliptic type.
\end{abstract}

\maketitle

\tableofcontents

\section{Introduction}

The Alexander polynomial of singular curves in $\PP^2_{\CC}$ 
provides an effective way to relate the fundamental group of the 
complement of such curves to the topology and geometry 
of their singularities. In this paper we show that the Mordell-Weil 
group of certain elliptic threefolds with constant $j$-invariant 
equal to either zero or 1728 is closely related to the Alexander 
module of plane curves associated with the threefolds. As a byproduct 
of this relation we obtain a bound on the degree of the Alexander 
polynomials of plane curves. The bound is linear in the degree of the 
curve and gives a new restriction on the groups which can be 
fundamental groups of the complements to plane curves. 

Let $\cC=\bigcup \cC_i$ be a curve in $\PP^2$ of degree $d$. Its Alexander polynomial is 
defined in terms of purely topological data:
 $G=\pi_1(\PP^2\setminus \cC_0\cup \cC)$, where $\cC_0$ is a line at infinity, 
and a surjection $\e: G \rightarrow \ZZ$ (cf. section~\ref{defalpol} 
for a definition). On the other hand, 
dependence of the Alexander polynomial of $\cC$, on its degree,  
the local type of its singularities, and their \emph{position}, 
relating the latter to the topology, 
 has been 
known for some time (cf.~\cite{Duke,position}).
For example, for irreducible curves having nodes and cusps as the only
singularities, the degree of the Alexander polynomial 
$\Delta_{\cC,\e}(t)$ equals $\rk G'/G''$ where $G'$ and $G''$ are 
respectively the first and second commutators of $G$ (in this case there is only one 
choice, up to sign, of $\e$). Moreover,
$\Delta_{\cC,\e}(t)$ is not trivial only if $6 \vert d$, in 
which case $\Delta_{\cC,\e}(t)=(t^2-t+1)^s$, where $s$ is the 
superabundance of the curves of degree $d-3-{d \over 6}$ passing through 
the cusps of $\cC$. For a given curve, this provides a purely geometric method
for calculation of the Alexander polynomial. 
However, how big can this superabundance be for a 
special cuspidal curve is still
not known (cf.~\cite{problems}). 
The largest known value of $s$ for irreducible cuspidal 
curves, to our knowledge, is 3.
This occurs for the dual curve to a non-singular cubic, that is a sextic
with 9 cusps. In this paper we give an example of an irreducible curve with
nodes and cusps as only singularities, for which the superabundance of the set
of cusps is equal to 4.

One of our main results is the inequality (cf.~Corollary~\ref{inequality}):
\begin{equation}\label{mainbound}
\deg \Delta_{\cC,\e} \le {5 \over 3}d-2.
\end{equation}

For reducible curves and general $\e$, 
the explicit relation with the fundamental group comes from 
the equality (cf.~\cite{Duke}):
\begin{equation}
\dim (G'/G''\otimes \QQ) \otimes_{\Lambda} 
{\Lambda}/{(t_1-t^{\e(\gamma_1)},\dots,
t_r-t^{\e(\gamma_r)})} 
=\deg \Delta_{\cC,\e},
\end{equation} 
where $\Lambda:=\ZZ[t_1^{\pm 1},\dots,t_r^{\pm 1}]$,
$r:=\rk(G/G')$ (if $\cC$ is a curve in $\CC^2$ then $r$ is 
the number of irreducible components of $\cC$), 
$\e:\ZZ^r\to \ZZ$ is an epimorphism, and $\gamma_i$ represents
the class of a meridian around the $i$-th irreducible component of $\cC$,
i.e. its abelian image in $H_1$ is $t_i$.

The main idea in this note is to relate the degree of the Alexander 
polynomial to the rank of the Mordell-Weil group over the field $\CC(x,y)$
of elliptic curves with $j$-invariant equal to either zero or 1728. 
This is done by relating the Alexander module of $G$ to rational 
pencils of elliptic type corresponding to $F$ (cf. Definition~\ref{def-orb}).

The relationship between Alexander invariants and pencils, 
in the case of line arrangements and reducible curves 
is discussed in~\cite{charvar,sergey}:
the positive dimensional components of characteristic varieties induce
maps between their complements and the complements to $p \ge 3$ points in $\PP^1$.
Here we show that non-vanishing of Alexander polynomial
of a curve yields existence of 
special pencils with non-reduced fibers. 
These pencils are such that
they induce rational maps from $\PP^2$ onto $\PP^1$ 
with an orbifold structure and take the 
curve $\cC$ onto a finite set of points. 
The pencils are orbifold elliptic pencils in the following sense
(cf. Definitions~\ref{def-orb} and~\ref{dfn-group-orb}): 
each map $\PP^2 \rightarrow \PP^1$ has three non-reduced fibers of the form 
$m_iD_i+F_i$, where $F_i$ divides the equation of $\cC$. The multiplicities $m_i$
are such that $\sum {1 \over {m_i}}=1$ and thus the orbifold structure is given by 
assigning multiplicities $m_i$ to those three points in $\PP^1$.

The possible orbifold structures $(m_1,m_2,m_3)$ depend on the local type of 
singularities of $\cC$.
For example for irreducible curves with only nodes and cusps as singularities, 
we relate the global Alexander polynomial of $\cC$ to the 
following functional relation: 
\begin{equation}\label{threefold}
P(x,y,z)^2+Q(x,y,z)^3+F(x,y,z)=0,
\end{equation}
where $F$ is an equation of $\cC$ and $P(x,y,z),Q(x,y,z)$ are homogeneous polynomials. 
Such a relation is equivalent to the existence
of a rational map from $\PP^2$ onto $\PP^1$ with the orbifold structure
$(2,3,6)$.
This orbifold structure can be obtained considering 
the global orbifold in the usual sense 
(cf.~\cite{ademruan}) corresponding 
to the action of a cyclic group of order 6 on 
an elliptic curve with non-trivial stabilizers at three points; 
their orders being equal to $2,3$ and $6$ respectively.

The correspondence between the Alexander modules and orbifold 
elliptic pencils is established in two rather different steps. On one hand,
the Alexander module of $\cC$ can be related to the Mordell-Weil group 
of the elliptic threefold:
\begin{equation}\label{threefold-2}
u^2+v^3=F(x,y,1)
\end{equation}
over the field $\CC(x,y)$ 
of rational functions in two variables having $j$-invariant equal to zero.
We have the following (cf. Theorem~\ref{theoremonWF}):

\begin{theorem} Let $\cC$ be an irreducible curve in $\PP^2$ having ordinary nodes 
and cusps as the only singularities. Let $F(x,y,1)=0$ be a (reduced) equation 
of the affine part of $\cC$. Then the $\ZZ$-rank of the Mordell-Weil group of the 
elliptic threefold (\ref{threefold-2})
is equal to the degree of the Alexander polynomial of the curve $\cC$.
\end{theorem}

The Mordell-Weil group here is the group of rational sections of the elliptic threefolds 
(see for instance~\cite{Langneron,Lang} and section~\ref{sec-mordellweil}
for further discussion).
The rank of the Mordell-Weil group of the threefold~\eqref{threefold-2}
was recently studied in~\cite{kloo} for the case $\deg \cC=6$
using different methods (cf. also~\cite{hulek}). These results follow 
immediately from the correspondence between Alexander polynomials 
and Mordell-Weil groups in this paper since the Alexander
polynomials of sextic curves considered in~\cite{kloo} are readily 
available.

On the other
hand, each element of the Mordell-Weil group of the aforementioned 
elliptic threefold defines a functional relation of the type~\eqref{threefold}.
This can be summarized as follows (cf. Theorem~\ref{cor-kulikov}):

\begin{theorem}
For any irreducible plane curve $\cC=\{F=0\}$ whose only singularities are nodes and cusps the
following statements are equivalent:
\begin{enumerate}
 \item
$\cC$ admits a quasi-toric relation of elliptic type $(2,3,6)$,
 \item
$\cC$ admits an infinite number of quasi-toric relations of elliptic type $(2,3,6)$,
 \item
$\Delta_{\cC,\e}(t)$ is not trivial
($\Delta_{\cC,\e}(t)\neq 1$) \footnote{for an irreducible curve,
there is only one choice of $\e$, up to sign, i.e.
$\Delta_{\cC,\e}$ is independent of it;
for other types of quasi-toric relations see section \ref{sec-examples}}.
\end{enumerate}

Moreover, the set of quasi-toric relations of $\cC$
$\{(f,g,h)\in \CC[x,y,z]^3\mid f^2+g^3+h^6F=0\}$ has a group structure and it is 
isomorphic to $\ZZ^{2q}$, where $\Delta_\cC(t)=(t^2-t+1)^q$. Also, $\cC$ admits 
an infinite number of primitive quasi-toric relations, unless $q=1$,
in which case $\cC$ only has one primitive quasi-toric relation.
\end{theorem}

We also consider here other singularities which will define relations of the form:
\begin{equation}\label{relationquasitoric}
h_1^pF_1+h_2^qF_2+h_3^rF_3=0,
\end{equation}
where $(p,q,r)$ is either $(3,3,3)$ or $(2,4,4)$. Such relations 
\eqref{relationquasitoric} in turn correspond to orbifold rational pencils 
with respective orbifold structures (cf. theorem \ref{cor-kulikov-gen}).

While commonly the existence of irrational pencils on surfaces is obtained 
by an extension of the deFranchis method (from 1-form with vanishing wedge product, 
cf.~\cite{defrancis,sernesi}), rational orbifold pencils are obtained rather differently:
here the pencils are a byproduct of the splitting of 
Albanese varieties of cyclic 
multiple planes into a product of elliptic curves 
which we derive either from Roan's Decomposition Theorem
for abelian varieties with an automorphism (cf.~\cite{Lange-Roan})
or directly, using a Hodge theoretical refinement of the argument 
used in the proof of divisibility theorem in \cite{Duke}.
For example for the cyclic multiple planes branched over a reduced
curve with nodes and cusps as the only singularities, the Albanese variety 
splits as a product of elliptic curves with $j$-invariant equal to zero.
One may contrast this with the case of Jacobians of curves where 
the situation is different (cf.~\cite{serre2}).

In particular, the relation between Alexander polynomials and 
Mordell-Weil groups allows us to give bounds~\eqref{mainbound}
on the degree of Alexander polynomials. These follow
from the bounds on the rank of Mordell-Weil groups obtained
from the connection with the Mordell-Weil groups of certain elliptic surfaces
and from the Shioda-Tate formula (cf.~\cite{shioda}).
However, the correspondence between the Alexander polynomials and the ranks 
of the Mordell-Weil group should be of independent interest (cf.~\cite{kloo}).

The curves with the largest 
known values of $\deg \Delta_\cC(t)$ are given in section~\ref{sec-examples}.
The bound presented here is sharp for sextics, however, the inequality~\eqref{mainbound} 
is apparently far from being sharp in general. Perhaps a better understanding of 
the Mordell-Weil rank of~\eqref{threefold-2} can yield a better estimate.
Also, note that Corollary~\ref{inequality} provides a partial 
answer to~\cite[Problem 2.1]{problems}.

Another application of the results presented in this paper is an alternative
argument to confirm Oka's conjecture of sextic curves having a
non-trivial Alexander polynomial (i.e. that equations of such curves have form
$P^2+Q^3$). The answer to Oka's conjecture was first obtained by 
A.~Degtyarev (cf.~\cite{degtyarev-okas-conjecture-i,degtyarev-okas-conjecture-ii}).

For the sake of clarity we often start our discussions focusing on the case of 
irreducible curves having only nodes and cusps as singularities. 
The results, however, are obtained, as 
was already mentioned, for curves with a wider class of singularities,
which we call $\delta$-essential and $\delta$-partial
(cf.~\ref{deltapartial}). Moreover the results 
are applicable to reducible and non-reduced curves as well.
From the point of view of fundamental groups, non-reduced curves 
correspond to homomorphisms $\e$ that are more general than those
given by the linking number of loops with~$\cC$.

The condition 
of being $\delta$-essential is purely local, meaning that the local 
Alexander polynomial of the link of the singularity considered w.r.t.
the restriction of $\e$ on the local fundamental group 
is divisible by the cyclotomic polynomial of degree~$\delta$.
As we shall see, this is the natural class of curves
leading to the elliptic pencils.

\subsection{Organization of the paper}
\mbox{}

The content of the paper is as follows. In section~\ref{sec-settings} we give
definitions for Alexander polynomials w.r.t. any homomorphism $\e$ (as mentioned 
above), for orbifold surfaces and morphisms, and for quasi-toric relations.
In section~\ref{sec-mordellweil} 
we relate the Alexander polynomial to the Mordell-Weil group of 
the threefold~(\ref{threefold-2})
associated with $F$. In section~\ref{sec-qt-rels}, the $\CC(x,y)$-points 
of the threefolds~\eqref{threefold} 
are presented as quasi-toric relations of $F$, i.e. functional equations of 
the form $f^2+g^3+Fh^6=0$, over the ring $\CC[x,y,z]$. The correspondence 
between the points of~\eqref{threefold} and the quasi-toric relations where 
$F$ fits in, is explained at the end of section~\ref{sec-mordellweil}. 
In section~\ref{sec-delta-curves} we generalize the results to a larger class
of curves. Finally, in section~\ref{sec-examples} a list of explicit examples 
and applications is given as to show curves that fit into quasi-toric relations 
of all rational orbifolds of elliptic type. Also one of these examples provides the largest 
values known to date of the degree of the Alexander polynomial of an irreducible
curve with nodes and cusps.

\subsection{Notation}
\label{ssec-notation}
\mbox{}

$F$ a (possibly reducible or even non-reduced) non-zero homogeneous polynomial in 
$\CC[x,y,z]$ such that $F$ is not a power, that is $G^k=F$ for $G\in \CC[x,y,z]$ implies $k=1$.

$\cC$ the set of zeroes of $F$, hence a reduced curve. 

$V_6$ a non singular model of a cyclic multiple plane $z^n=F(x,y)$

$\bar V_6$ a model of cyclic multiple plane in $\PP^2 \times \PP^1$.

$W_F^{\circ}$ an affine model of the elliptic threefold corresponding to a curve 
$F=0$.

$W_F$ a smooth projective birational model of $W_F^{\circ}$.

$\tilde W_F$ singular projective model of $W_F$ (cf. Proposition~\ref{equationforw}).

$E_0$ the elliptic curve with j-invariant zero.

$\bar E_Q$ a model of $E_0$ in $\PP(2,3,1) \times \PP^1$.

$W$ a split elliptic threefold $V_6 \times E_0$.

$\bar W=\bar V_6 \times_{\PP^1} \bar E_Q$ a birational model 
of split elliptic threefold $V_6 \times E_0$

%

\section{Preliminaries}
\label{sec-settings}

In this section we will review several results on Alexander invariants 
which appear in the literature and extend them to the generality 
required in this paper.

\subsection{Alexander polynomial relative to a surjection of the fundamental group}
\label{defalpol}

We shall consider reducible, not necessarily reduced curves in $\PP^2$.
Let $\cC$ be a plane curve given as the set of zeroes of a homogeneous polynomial $F=F_1^{e_1}\cdots F_r^{e_r}$,
which is not a power (see Notation~\ref{ssec-notation}), which means $\gcd(\e_1,\dots,\e_r)=1$. 
Consider $\cC:=\cC_1 \cup \dots \cup \cC_r$ its decomposition into irreducible components and let 
$(\e_1,\dots,\e_r)$, $\e_i \in \ZZ_+$ be the collection of multiplicities of each irreducible factor $F_i$ in 
the equation of $\cC$. So if $d_i=\deg F_i$ denotes the degree of $\cC_i$, then the 
total degree of $\cC:=\{F=0\}$ is given by $d:=\sum \e_id_i=\deg F$.

Let $\cC_0$ be a line transversal to $\cC$ which we shall view as the line 
at infinity and let $G:=\pi_1(\PP^2\setminus \cC_0\cup \cC)$.
Recall that $H_1(\PP^2\setminus \cC_0\cup \cC)$ is a free 
abelian group with $r$ generators having  canonical identification 
with $\ZZ^r$
(cf.~\cite{charvar}). The isomorphism is given by mapping 
the class of the boundary of a small holomorphic 2-disk transversal to 
the component $\cC_i$ to 
$(0,..,1,...,0) \in \ZZ^r$ (with $1$ appearing as the $i$-th component).
Let $\e$ be the epimorphism 
$\e: G\to H_1(\PP^2\setminus \cC_0\cup \cC) \rightarrow \ZZ$, 
given by $\e(\gamma_i):=\e_i$. Let $\QQ[\ZZ]=\QQ[t,t^{-1}]$ 
denote the group ring over $\QQ$, $K_{\e}=\ker \e$, and 
$K_{\e}'=[K_{\e},K_{\e}]$ be the commutator of $K_{\e}$. By the Hurewicz Theorem
$K_{\e}/K_{\e}'$ can be identified with the homology of infinite cyclic cover 
of $\PP^2\setminus \cC_0\cup \cC$ corresponding to $\e$
and hence can be viewed as a module 
over the group ring $\ZZ[\ZZ]$.
 
Homomorphisms (resp. epimorphism) $\e: \pi_1(\PP^2\setminus \cC_0\cup \cC)\rightarrow \ZZ$ 
such that $\e(\gamma_i) \ge 1$ for all $i$, are in one-to-one correspondence with polynomials, 
(resp. non-power polynomials) considered up to a scalar factor, having $\cC$ as 
the zero set. Indeed if $F_i$ is an irreducible polynomial having $\cC_i$ as its 
set of zeros, then one defines the polynomial corresponding to $\e$ as 
\begin{equation}\label{productepsilon}
F_{\e}(x,y,1)=\Pi_i F_i(x,y,1)^{\e (\gamma_i)}. 
\end{equation}
Vice versa, given a polynomial $F$ having $\cC$ as
its zero set one defines the homomorphism
$\e_F:H_1(\PP^2 \setminus \cC_0\cup \cC) 
\rightarrow \ZZ$ using
\begin{equation}\label{definingepsilon}
\e_F (\gamma_i)={1 \over {2 \pi \sqrt{-1} }}
\int_{\gamma_i} {{dF}\over F}
\end{equation}
and extends it as the composition
$\pi_1(\PP^2 \setminus \cC_0\cup \cC) \rightarrow 
H_1(\PP^2 \setminus \cC_0\cup \cC) \rightarrow \ZZ$. 
This homomorphism is in fact an epimorphism if $F$ is not a power. It will also be denoted by $\e_F$.
 
\begin{dfn}(cf.~\cite{Duke,oka-survey}) 
The Alexander polynomial $\Delta_{\cC,\e}(t)$
of $\cC$ relative to a surjection $\e: G \to \ZZ$ is a generator of the order of the torsion
of the $\QQ[\ZZ]=\QQ[t,t^{-1}]$-module $K_{\e}/K_{\e}'\otimes \QQ$ normalized in such a way that 
it is a polynomial in $t$ satisfying $\Delta_{\cC,\e}(0)=1$.
By the discussion in the previous paragraph, a (not necessarily reduced) equation $F$ of $\cC$ 
defines both the set of zeroes $\cC$ and a surjection $\e_F$. Hence $\Delta_{F}(t)$ will
also denote $\Delta_{\cC,\e_F}(t)$.
\end{dfn}
 
The Alexander polynomial $\Delta_{\cC,\e}$ can be expressed in terms 
of characteristic varieties studied in~\cite{charvar} or in terms of $G'/G''$ 
viewed as a module over 
$\Lambda:=\ZZ[H_1]=\ZZ[t_1^{\pm 1},\dots,t_r^{\pm 1}]$ as follows.
The polynomial $\Delta_{\cC,\e}$ is the order of the torsion of 
$$
(G'/G'' \otimes \QQ) \otimes_{\Lambda}\Lambda/ 
{(t_1-t^{\e(\gamma_1)},\dots,
t_r-t^{\e(\gamma_r)})} 
$$
up to a power of $(t-1)$, viewed as a $\QQ[t,t^{-1}]$-module in the obvious way.
The zeroes of $\Delta_{\cC,\e}$ can also be seen as the intersection of 
the characteristic variety $\Sigma_1(\cC)$ with the 1-dimensional torus of equation 
$L_\e:=\{(t^{\e(\gamma_1)},\dots,t^{\e(\gamma_r)})\}
\subset (\CC^*)^r$ (cf.~\cite[Theorem 2.26]{kike-hiro-survey}).
\footnote{Though most often the Alexander polynomials are considered 
in the case when $\e_i=1$, case $\e_i \ne 1$ was considered for example by 
Oka in~\cite[\S4]{oka-survey} as \emph{$\theta$-Alexander polynomials}.}
 
We shall also need the local version of the polynomials 
$\Delta_{\cC,\e}$ defined similarly.
Let $P$ be a singular point of $\cC$.
The epimorphism $\e: G \rightarrow \ZZ$ induces a homomorphism 
$\e_P$ of the local fundamental group of $\cC$ to $\ZZ$ 
and hence the Alexander polynomial of the link of $P$ 
w.r.t. the homomorphism $\e_P$
\footnote{If the image of $\e_P$ has index $k$ in $\ZZ$, then
the Alexander polynomial $\Delta_{\cC,\e,P}(t)$ is $\Delta(t^k)$ where $\Delta$ is 
the Alexander polynomial relative to the surjection $G \ \rightmap{\e_P}\ \e_P(G)=\ZZ$.
Recall that this represents the order of the 1-dimensional homology of the cyclic 
cover corresponding to $\e$ which is this case has $k$ 
connected components cf.~\cite{Duke,charvar}}. 
In other words, if 
$i:\SSS_P\setminus \cC\hookrightarrow \PP^2\setminus \cC$ is the inclusion
from a sufficiently small sphere around $P$ in the total space and 
$\gamma$ is a meridian around a component of the link, then
\begin{equation}
\label{eq-defep} 
\e_P(\gamma):=\e(i_*(\gamma)).
\end{equation}
These polynomials will be denoted by $\Delta_{\cC,\e,P}(t)$.
We have the following:
 
\begin{prop}
\label{prop-divisibility}
Let $\cC$ be a plane curve and $(\e_1,\dots,\e_r)$ denote the multiplicities of 
its irreducible. Then $\Delta_{\cC,\e}(t)$ divides the product of the local Alexander polynomials:
\begin{equation}\label{firstdiv}
\Delta_{\cC,\e}(t) \vert 
\prod_P \Delta_{\cC,\e,P}(t)
\prod (t^{\e_i}-1)^{k_i}.
\end{equation}
In particular, if the local Alexander polynomials have only roots of 
unity of degree $\delta$ as their roots, then:
$$\Delta_{\cC,\e}(t)=
\prod_{\lambda\vert \delta} \varphi_\lambda(t)^{s_\lambda}\prod (t^{\e_i}-1)^{k_i},$$
where $\varphi_\lambda(t)$ is the cyclotomic polynomial of the $\lambda$-th 
roots of unity. Moreover, if $s_\lambda>0$, then 
\begin{equation}\label{seconddiv}
\lambda \vert d,
\end{equation}
where $d:=\sum d_i\e_i$ is the total degree.
\end{prop}
 
\begin{proof} 
Details of the arguments are similar to those used in the proof of 
the Divisibility Theorem (cf.~\cite{Duke,position} and Lemma~\ref{albaneseeigenvalues} below).
The starting point is the surjection of the fundamental group of a regular neighborhood 
of $\cC$ in $\PP^2\setminus \cC_0$ onto $\pi_1(\PP^2\setminus \cC_0 \cup \cC)$ (which is 
a consequence of the Lefschetz Hyperplane Section Theorem).
On the other hand, one uses the Mayer-Vietoris sequence for the 
$\e$-cyclic cover of this neighborhood in order to split it into a union of cyclic 
covers: those of the local singularities and the neighborhood of the non-singular 
part of $\cC$. This shows that the Alexander polynomial of the neighborhood 
of $\cC$ is equal to the 
product of the Alexander polynomials of singularities and a divisor of second product in 
\eqref{firstdiv}. These divisors come as contributions of $H_0$ and $H_1$ of terms of the 
Mayer-Vietoris sequence corresponding to the intersections of the complements to the links 
of the singularities with the mentioned $\CC^*$-bundle of the non-singular part of $\cC$. 
This yields the divisibility~\eqref{firstdiv}.
 
The second divisibility relation is a generalization of the 
divisibility at infinity (cf.~\cite{Duke}) and follows from the 
calculation of the Alexander polynomial of the link 
at infinity with multiplicities. It is equal to 
$(t^{\sum \e_id_i}-1)^{r-1}$
as a consequence of the Torres relation (cf.~\cite{turaev}) applied 
to the Hopf link with multivariable Alexander polynomial 
$(t_1 \cdot ... \cdot t_r-1)^{r-1}$.
\end{proof}
 
\subsection{Orbifold curves}\label{orbifoldcurves}
Now we recall some basic definitions needed here 
referring for more details to (see~\cite{ACM-multiple-fibers}).
 
\begin{dfn}
\label{def-orb}
An \emph{orbifold} curve $S_{\bar m}$ is (an open or closed) 
 Riemann surface~$S$
with a function $\bar m:S\to\NN$ whose value is~$1$ outside a finite number of points.
A point $P\in S$ for which $\bar m(P)>1$ is called an \emph{orbifold point}.
\end{dfn}
 
One may think of a neighborhood of a point $P\in S_{\bar m}$ with $\bar m(P)=d$ as the quotient of a disk 
(centered at~$P$) by a rotation of angle $\frac{2\pi}{d}$. A loop around $P$ is considered to be 
trivial in $S_{\bar m}$ if its lifting bounds a disk. Following this idea, orbifold fundamental groups
can be defined as follows.
 
\begin{dfn}\label{dfn-group-orb}
For an orbifold $S_{\bar m}$, let $P_1,\dots,P_n$ be the orbifold points, 
$m_j:=\bar m(P_j)>1$. Then, the \emph{orbifold fundamental group} of $S_{\bar m}$ is
\begin{equation}
\label{eq-orb-fg}
\pi_1^{\text{\rm orb}}(S_{\bar m}):=\pi_1(S\setminus\{P_1,\dots,P_n\})/\langle\mu_j^{m_j}=1\rangle,
\end{equation}
where $\mu_j$ is a meridian of $P_j$. We will denote $S_{\bar m}$ simply by $S_{m_1,\dots,m_n}$.
\end{dfn}
 
\begin{remark}
In this paper, we will mostly consider orbifold groups of $\PP^1$ with three
orbifold points. The groups 
$\pi_1^{\text{orb}}(\PP^1_{(p,q,r)})=\langle x,y,z: x^p=y^q=z^r=xyz=1 \rangle$ 
are subgroups of index two of full triangle groups. In particular, they can 
be identified with the orientation-preserving isometries of a plane tiled with 
triangles with angles $\frac{\pi}{p}$, $\frac{\pi}{q}$, and $\frac{\pi}{r}$ 
(cf.~\cite[Corollary 2.5]{milnor-brieskorn-manifolds}).
\end{remark}
 
\begin{dfn}
A dominant algebraic morphism $\varphi:X\to S$ defines an \emph{orbifold morphism} 
$X\to S_{\bar m}$ if for all $P\in S$, the divisor $\varphi^*(P)$ is a 
$\bar m(P)$-multiple.
\end{dfn}
 
One has the following result regarding orbifold morphisms.
 
\begin{prop}[{\cite[Proposition~1.5]{ACM-multiple-fibers}}]
\label{prop-orb}
Let $\rho:X\to S$ define an \emph{orbifold morphism} $X\to S_{\bar m}$. Then $\varphi$ induces a morphism
$\varphi_*:\pi_1(X)\to\pi_1^{\text{\rm orb}}(S_{\bar m})$. Moreover, if the generic fiber is connected, then
$\varphi_*$ is surjective.
\end{prop}
 
Proposition~\ref{prop-orb} will be applied systematically throughout this paper.
We will show a typical example of this. Suppose $F=F_1F_2F_3$ fits in a functional equation of type 
\begin{equation}
\label{eq-qtoric} 
h_1^3F_1+h_2^2F_2+h_3^6F_3=0,
\end{equation} 
where $h_1, h_2$, and $h_3$ are polynomials. 
Note that $F_i$ are not necessarily irreducible or reduced. 
Also note that~\eqref{eq-qtoric} induces a pencil 
map $\varphi:\PP^2\dasharrow \PP^1$ given by $\varphi([x:y:z])=[h_1^3F_1:h_2^2F_2]$.
Consider $\cC:=\{F=0\}$. Since $\varphi|_{\PP^2\setminus \cC}$ has three multiple
fibers (over $P_3=[0:1]$, $P_2=[1:0]$, and $P_6=[1:-1]$) one has an orbifold morphism 
$\varphi_{2,3,6}:\PP^2\setminus \cC\to \PP^1_{2,3,6}$. 
In particular, if the pencil $\varphi$ is primitive
(in the sense that it coincides with its Stein factorization) 
then by Proposition~\ref{prop-orb}, there is an epimorphism 
$$\varphi_{2,3,6}:\pi_1(\PP^2\setminus \cC)\to \pi_1^{\text{\rm orb}}(\PP^1_{2,3,6})= 
\frac{\mu_2\ZZ_2*\mu_3\ZZ_3}{(\mu_2\mu_3)^6},$$ 
where $\mu_i$ are as in Definition \ref{dfn-group-orb} and
$\mu_2\mu_3=\mu_6$ according to~\eqref{eq-orb-fg}. Finally, note that 
$V:=\Char\left(\frac{\mu_2\ZZ_2*\mu_3\ZZ_3}{(\mu_2\mu_3)^6}\right)= \{\omega_6, \omega_6^{-1}\}$, 
and the elements $\{\omega_6^{\pm 1},\omega_6^{\pm 2},\omega_6^{\pm 3}\}$ are roots 
of $\Delta_{\cC,\e}(t)$, where
\begin{equation} 
\label{eq-e-qt} 
\e:H_1(\PP^2\setminus \cC_0 \cup \cC) \to
\frac{\pi_1^{\text{\rm orb}}(\PP^1_{2,3,6})} 
{[\pi_1^{\text{\rm orb}}(\PP^1_{2,3,6}),\pi_1^{\text{\rm orb}}(\PP^1_{2,3,6})]}=\ZZ_6 
\end{equation} 
is induced by $\varphi_{2,3,6}$ on the abelianizations of the groups. 
Therefore one has that $\Delta_{\cC,\e}(t)$ is of the form 
$$(t-1)^{s_1}(t+1)^{s_2}(t^2+t+1)^{s_3}(t^2-t+1)^{s_6}p(t),$$ 
where $p(t)$ has no 6-th roots of unity as zeroes and $s_i$ are
non-negative integers. 
 
Using this technique one can show the following result 
needed in the sequel and which we shall prove for completeness 
(cf. \cite{friedman-morgan-smooth}, Ch 2, Theorem 2.3). 
 
\begin{prop}
\label{prop-bound-pencil}
The number of multiple members in a primitive pencil of plane curves (with no base components) 
is at most two.
\end{prop}
 
\begin{proof}
Let us assume that the pencil is generated by two multiple fibers, that is, 
$\varphi:\PP^2\dasharrow \PP^1$, given by $\varphi([x:y:z])=[f^p:g^q]$, where $(p,q)=1$ 
(otherwise, the pencil is not primitive). Assume there is a third multiple member, that is, 
$f^p+g^q+h^r=0$.
 
According to Proposition~\ref{prop-orb}, one obtains an orbifold morphism with connected 
fibers, and thus an epimorphism 
$\varphi_*:\pi_1(\PP^2\setminus \{p_1,...,p_n\})\to \pi_1^{\text{\rm orb}}(\PP^1_{p,q,r})$ 
i.e. both groups are trivial.
The groups $\pi_1^{\text{\rm orb}}(\PP^1_{p,q,r})$ are virtually torsion 
free (cf.~\cite[Theorem 2.7]{milnor-brieskorn-manifolds}) and they are 
subgroups of order two of the Schwartz group, which is infinite
(cf.~\cite[Corollary 2.4]{milnor-brieskorn-manifolds}).
In particular, they are non-trivial, which contradicts the
triviality of $\pi_1(\PP^2\setminus \{p_1,...,p_n\})$.
\end{proof}
 
\begin{remark}
For pencils other than pencils of plane curves,
using logarithmic transforms, one can obtain
elliptic fibrations with any number of multiple fibers
(cf.~\cite{friedman-morgan-smooth,kodaira-structure-i}).
\end{remark}
 
In this paper we are interested in a particular type of orbifold morphisms.
 
\begin{dfn} 
We say a 2-dimensional orbifold $S_{\bar m}$ is a 
\emph{rational orbifold curve of elliptic type} if $e_{\text{top}}(S_{\bar m})=2$ 
(that is $S_{\bar m}$ is a compact Riemann sphere) and 
$$\sum \frac{1}{m_i}=n-2,$$
where $n:=\#\{P\in S \mid \bar m(P)>1\}$ is the number of orbifold points.
\end{dfn}
 
\begin{lemma} 
\label{lemma-elliptic-type}
The possible rational orbifold curves of elliptic type are:
\begin{enumerate}
 \item $(2,3,6),$
 \item $(3,3,3),$
 \item $(2,4,4),$ and
 \item $(2,2,2,2).$
\end{enumerate}
\end{lemma}
 
\begin{remark}
\label{rem-elliptic}
The reason to call such orbifolds of elliptic type is the following. Consider 
$\hat S_{\bar m}$ the regular covering of order $\ell:=\lcm (\bar m)$ 
(the least common multiple of the orbifold multiplicities) ramified with index
$m_i$ at the $\frac{\ell}{m_i}$ preimages of $P_i$ (the orbifold point of order $m_i$)
of $S_{\bar m}$. In general, if $\sum \frac{1}{m_i}\in \ZZ$, then $\hat S_{\bar m}$ is a
Riemann surface of genus 
$$
1+\frac{\ell}{2}\left(\sum (n-2) - \frac{1}{m_i} \right).
$$
Thus, according to Lemma~\ref{lemma-elliptic-type}, for any rational orbifold curve 
$\hat S_{\bar m}$ of elliptic type the associated covering $\hat S_{\bar m}$ is an 
elliptic curve (that is, a complex compact curve of genus 1).
\end{remark}

The results in this paper require that the least common multiple of the orbifold
indices be $>2$. Therefore type $(4)$ from the list above will be disregarded.

\subsection{Quasi-toric relations}\label{quasitoricsection}

\begin{dfn} 
A \emph{quasi-toric relation of type $(p,q,r)$} is a sextuple
$\cal R^{(p,q,r)}_{{\rm qt}}:=(F_1,F_2,F_3,h_1,h_2,h_3)$ of non-zero homogeneous polynomials in $\CC[x,y,z]$
satisfying the following functional relation
\begin{equation}
\label{eq-qt-rel}
h_1^pF_1+h_2^qF_2+h_3^rF_3=0.
\end{equation}

The \emph{support} of a quasi-toric relation $\cR^{(p,q,r)}_{{\rm qt}}$ as above is the 
zero set $\cC:=\{F_1F_2F_3=0\}$. In this context, we may also refer to $\cC$ as a curve that 
\emph{satisfies (or supports) a quasi-toric relation} of type $(p,q,r)$.

We will say a quasi-toric relation of type $(p,q,r)$ is \emph{of elliptic type} 
if $(p,q,r)$ is of the form $(1)-(3)$ in Lemma~\ref{lemma-elliptic-type}.
\end{dfn}

\begin{remark}
Given a quasi-toric relation $\cR^{(p,q,r)}_{{\rm qt}}=(F_1,F_2,F_3,h_1,h_2,h_3)$ as above, 
note that
\begin{equation}
\label{eq-kappa1}
p \deg h_1 + \deg F_1 = q \deg h_2 + \deg F_2 = r \deg h_3 + \deg F_3 = \kappa
\end{equation}
and hence 
$$
\sum \deg h_i + \left( \frac{\deg F_1}{p} + \frac{\deg F_2}{q} + \frac{\deg F_3}{r}\right) = 
\left( \frac{1}{p} + \frac{1}{q} + \frac{1}{r}\right) \kappa.
$$
Therefore, if $\cR^{(p,q,r)}_{{\rm qt}}$ is of elliptic type, then 
\begin{equation}
\label{eq-kappa2}
\sum \deg h_i + \left( \frac{\deg F_1}{p} + \frac{\deg F_2}{q} + \frac{\deg F_3}{r}\right) = \kappa.
\end{equation}
\end{remark}

\begin{remark} Special classes of curves satisfying quasi-toric relations have already been considered,
namely, the class of \emph{curves of torus type} $(2,3)$, i.e. quasi-toric relations of type 
$(2,3,r)$ of the form $(1,1,F,h_1,h_2,1)$ (cf.~\cite{oka-p-q,pho-torus}) and the class of 
\emph{curves having quasi-toric decompositions} $(p,q,pq)$, i.e. quasi-toric relations of type 
$(p,q,pq)$ of the form $(1,1,F,h_1,h_2,h_3)$ (cf.~\cite{kulikov-albanese}).
Also, functional equations of a similar sort appear in~\cite[\S5.3]{oka-tangential}. 
 
Recall that a \emph{quasi-toric decomposition of $F$} is a collection of homogeneous 
polynomials $f,g,h \in \CC[x,y,z]$ such that the following identity holds: 
\begin{equation} 
\label{eq-qt} 
f^p+g^q+h^{pq}F=0, 
\end{equation}
for two co-prime positive integers $p,q>1$. Analogously, a curve is of torus type $(p,q)$ 
if it admits a quasi-toric decomposition as in~\eqref{eq-qt} where $h=1$.

\end{remark}

To each quasi-toric relation $\cR^{(p,q,r)}_{{\rm qt}}$ of elliptic type
satisfying~\eqref{eq-qt-rel} there corresponds a map from a certain cyclic multiple plane 
branched over the curve supporting $\cR^{(p,q,r)}_{{\rm qt}}$ to the elliptic curve 
$\hat S_{(p,q,r)}$ (see Remark~\ref{rem-elliptic}) as follows. Let $\ell:=\lcm (p,q,r)$  
and $\omega={{\deg F_1} \over p}+{{\deg F_2} \over q}+{{\deg F_3} \over r}$
(note that according to~\eqref{eq-kappa2} one has $\omega=\kappa - \sum \deg h_i \in \ZZ$).
Let $V_\ell$ be the following surface given in the weighted projective space $\PP^3(\omega,1,1,1)$ 
\begin{equation}
\label{eq-qhomo}
V_\ell := \{(u,x,y,z)\in \PP^3(\omega,1,1,1) \mid u^\ell=F_1^{\ell \over p} F_2^{\ell \over q}F_3^{\ell \over r}\}.
\end{equation}
Then let $\hat S_{(p,q,r)}$ be the elliptic curve in $\PP^2$ given by the equation 
$z^\ell=x^{\ell \over p}y^{\ell \over q}(-x-y)^{\ell \over r}$.
To each quasi-toric relation there corresponds the following map
\begin{equation}
\label{eq-map}
\array{rcl}
\PP^3(\omega,1,1,1) \supset V_\ell & \rightarrow & \hat S_{(p,q,r)}\subset \PP^2\\
(u,x,y,z) & \mapsto & (h_1^pF_1,h_2^qF_2,uh_1h_2h_3).
\endarray
\end{equation}
(Note that the map is well defined by~\eqref{eq-kappa1} and~\eqref{eq-kappa2}).

\begin{dfn}
We will say that two quasi-toric relations, $(h_1,h_2,h_3,F_1,F_2,F_3)$ and 
$(\bar h_1,\bar h_2,\bar h_3,\bar F_1,\bar F_2, \bar F_3)$ of the same elliptic type $(e_1,e_2,e_3)$
are \emph{equivalent} iff the corresponding maps~\eqref{eq-map} coincide i.e. there exists
a non-trivial rational function $\lambda \in \CC(x,y,z)^*$ such that 
\begin{equation}
\label{eq-qt-equiv}
\array{l}
\bar h_i^{e_i}\bar F_i = \lambda h_i^{e_i}F_i \ \ \text{ and } \\ \\
\bar h_1\bar h_2\bar h_3 = \lambda h_1h_2h_3.
\endarray
\end{equation}
\end{dfn}

\begin{example}
\label{ex-236}
Note that any quasi-toric relation $(F_1,F_2,F_3,h_1,h_2,h_3)$ of elliptic type 
$(2,3,6)$ is equivalent to one of the form $(1,1,F_3F_2^2F_1^3,\bar h_1,\bar h_2,\bar h_3)$
since $\lambda:=F_1^3F_2^2$, $\bar h_1:=F_1^2F_2h_1$, $\bar h_2:=F_1F_2h_2$, and $\bar h_3:=h_3$ 
satisfy~\eqref{eq-qt-equiv}.

In other words, any $(2,3,6)$ quasi-toric relation is equivalent to one of the form:
$$
\bar h_1^2 + \bar h_2^3 + F_1^3F_2^2F_3 \bar h_3^6 = 0.
$$
This stresses the idea that non-reduced components are indeed required.
\end{example}

\section{Mordell-Weil group of elliptic threefolds  
with fiber having $j=0$}\label{sec-mordellweil}

\subsection{Elliptic pencils, rank of Mordell-Weil group, and 
the degree of Alexander polynomials}
\label{subsec-MW}
In this section, let us fix $F\in \CC[x,y,z]$, an irreducible homogeneous polynomial of 
degree $d=6k$, whose set of zeroes in $\PP^2$ is a curve $\cC$ that has only nodes and cusps 
as singularities (i.e. $\cC$ has either $x^2=y^2$ or $x^2=y^3$ as a local equation around each 
singular point). The homomorphism $\e: \pi_1(\PP^2 \setminus \cC_0 \cup \cC) \to \ZZ$, which is 
given by the total linking number with $\cC$ in the affine plane $\PP^2\setminus \cC_0$, in this
case corresponds to the trivial morphism $\e(\gamma)=1$ for any meridian around $\cC$.

Consider a threefold $W_F$ containing, as a Zariski open subset, 
the affine threefold $W_F^{\circ}$ given in
$\CC^4$ by the following equation:
\begin{equation}
\label{equationWF}
u^2+v^3=F(x,y,1).
\end{equation}
The projection onto the $(x,y)$-plane
exhibits $W_F$ as an elliptic threefold whose fibers
over generic points have $j$-invariant equal to zero 
\footnote{Since we are interested in the Mordell-Weil group, which is a birational invariant of $W_F$, 
the actual choice of $W_F$ is not important. Nevertheless, a particular biregular model 
of $W_F$ will be given (cf.~Lemma \ref{equationforw})}

The main result of this section is a calculation of the Mordell-Weil group of 
the $\CC(x,y)$-points of $W_F$ in terms of the classical Alexander polynomial $\Delta_{\cC}(t)$ 
of $\cC$ (that is, the Alexander polynomial w.r.t. the homomorphism $\e$ described above). 

\begin{theorem}\label{theoremonWF}
The $\ZZ$-rank of the Mordell-Weil group of $W_F$ over $\CC(x,y)$ is equal to the degree 
$\deg (\Delta_{\cC}(t))$ of the Alexander polynomial of $\cC$.
\end{theorem}

Let $V_6(F)$ (or simply $V_6$)
denote a smooth model the 6-fold cyclic cover of $\PP^2$ branched along $F=0$
corresponding to the surjection $\pi_1(\PP^2 \setminus \cC_0 \cup \cC)
\rightarrow \ZZ_6$ which is defined using the composition of the 
homomorphism as above and the reduction modulo $6$.
$V_6(F)$ contains, as an open subset, the affine hypersurface in $\CC^3$ 
given by the equation 
\begin{equation}\label{multipleplane}
z^6=F(x,y,1). 
\end{equation}
We shall assume that $V_6(F)$ supports a 
holomorphic action of the group $\mu_6$ of roots of unity of degree 6, 
extending the action of this group 
affine surface given by $(x,y,z) \rightarrow (x,y,\omega_6z)$ 
where $\omega_6=exp({{2 \pi \sqrt{-1}} \over 6})={{1+\sqrt{-3}}\over 2}$
is selected 
primitive root of unity of degree 6. Such smooth model $V_6$ can be obtained 
for example using an equivariant resolution of singularities of projective 
closure of (\ref{multipleplane}). 
 
Recall that the degree of the Alexander polynomial $\Delta_{\cC}(t)$ of $\cC$ 
is equal to $2q$, where $q$ is the irregularity 
$q:=\dim H^0(V_6,\Omega^1_{V_6})$ of $V_6$ (cf.~\cite{Duke}). 
Let $E_0$ denote the elliptic curve with $j$-invariant equal to zero.
As its biregular model one can take the projective closure of
affine curve $u^2+v^3=1$. 
 
We shall start the proof of Theorem~\ref{theoremonWF} 
by describing the Albanese varieties of six-fold cyclic multiple planes $V_6$. 
Such a description will be based on the following Decomposition Theorem due to S.~Roan 
(see~\cite[Theorem 3.2]{Lange-Roan}).
 
\begin{theorem}\label{roanth}
Let $X$ be an abelian variety of dimension $g$ and $\alpha$ an automorphism of $X$ of order $d\geq 3$. 
Let $\Phi_{\alpha}$ be the collection of the eigenvalues of the automorphism $\rho_{\alpha}$ induced by 
$\alpha$ on the universal cover of $X$ and let $X_{\langle \alpha \rangle}$ 
be the union of fixed points of powers $\alpha^i$ for $1 \le i <d$. 
Assume that $X_{\langle \alpha \rangle}$ is finite and 
$\# \Phi_\alpha=\frac{\varphi(d)}{2}$, where $\varphi$ is the Euler function. Then $\varphi(d)\mid 2g$
and $(\QQ(\xi_d),\Phi_\alpha)$ is a CM-field
(cf.~\cite[Section~3]{Lange-Roan}).
Moreover, there are $k=\frac{2g}{\varphi(d)}$ abelian varieties
$X_1, ... , X_k$ of CM-type $(\QQ(\xi_d),\Phi_\alpha)$ s.t. $X$ factors as
$$
X_1\times ... \times X_k
$$
and $\alpha$ decomposes into a product of automorphisms of $X_i$ induced by 
primitive $d$-th roots of unity.
\end{theorem}

Recall (cf.~\cite{serre}) that the Albanese variety of a smooth projective variety $X$ is defined as 
\begin{equation}\label{defalb}
\Alb(X)=H^0(X,\Omega^1_X)^*/H_1(X,\ZZ)
\end{equation}
and the Albanese map 
$alb: X \rightarrow \Alb(X)$ is given by
\begin{equation}\label{albanesemap}
P \rightarrow \int_{P_0}^P \omega,
\end{equation} 
where $P_0$ is a base point and the integral in~(\ref{albanesemap}) is viewed as a linear 
function on $H^0(X,\Omega^1_X)$, well defined up to the periods (i.e. the 
integrals over loops representing classes in $H_1(X,\ZZ)$). The map $alb$ is well defined 
up to translation (i.e. a choice of $P_0$). A choice of a positive line bundle on $X$ 
yields a polarization of $\Alb(X)$ making it into an abelian variety. 
In the case of 
the cyclic multiple plane $V_6$ we select a point $P_0$ in the 
locus with maximal ramification index 
of its projection onto $\PP^2$. The map $alb$ is universal in the sense that for any map 
$X \rightarrow A$ into an abelian variety $A$, there exists a factorization:
\begin{equation}
X \rightarrow \Alb(X) \rightarrow A.
\end{equation}
It follows from (\ref{defalb}) that when $X$ carries a biholomorphic action of a group 
$G$ fixing the base point $P_0$, the action of $G$ on $H^0(X,\Omega^1_X)$,
given by $g(\omega)=({g^{-1}})^*\omega$, induces the action 
on $\Alb(X)$. For this action the map $alb$ 
is equivariant:
\begin{equation}
alb(g\cdot P)=g_*^{-1}(alb(P)).
\end{equation}
For example, for $g \in G$ one has $g_*^{-1}alb(P)=
\int_{P_0}^P(g^{-1})^*\omega=\int_{gP_0}^{gP}\omega=alb(g\cdot P)$.

We will need also a local version of the above construction corresponding 
to the mixed Hodge structure associated with 
a germ of plane curve singularity (with assumptions stated below): 
\begin{equation}\label{planesing}
f(x,y)=0. 
\end{equation}
While there are  several constructions 
of a mixed Hodge structure 
associated with a germ of singularity 
(\ref{planesing}) (cf.~\cite{Steen}), we shall 
consider only the case
when the monodromy action on the cohomology of Milnor fiber is 
 semi-simple (e.g. the ordinary cusps, nodes and 
more generally the singularities appearing in tables~\ref{tab:d-ess},
\ref{tab:d-4-ess} and~\ref{tab:d-3-ess}
 in section~\ref{sec-delta-curves}).
In this case one can identify the (co)homology of the Milnor fiber with the
(co)homology of the link of the surface singularity 
\begin{equation}\label{localcover}
z^N=f(x,y), 
\end{equation}
where $N$ is the order of the monodromy of the cohomology of Milnor fiber
(cf. for example~\cite{alexhodge} for a similar discussion).
More precisely we have the following:

\begin{lemma}\label{comparemhs} Let (\ref{planesing}) be a germ 
of a plane curve
(possibly reducible and non-reduced) 
with semi-simple monodromy of order $N$ and the Milnor fiber $F_f$. Let $L_{f,N}$ be link 
of the corresponding surface singularity (\ref{localcover}).
Then there is the isomorphism of the mixed Hodge structures:
\begin{equation}\label{lemmaisomorphism}
H^2(L_{f,N})(1)=H^1(F_f)
\end{equation}
where the mixed Hodge structure on the left is the Tate twist of 
the mixed Hodge structure constructed in 
\cite{durfee1} or  \cite{durfee} 
 and the one on the right is the mixed Hodge structure 
on vanishing cohomology constructed in \cite{van}
\end{lemma}

Recall that the construction of the mixed Hodge structure on the 
cohomology of a link of an isolated singularity is based on 
the identification of the latter with the cohomology 
of the punctured regular neighborhood of the exceptional set in a 
resolution for~(\ref{localcover}). 
Alternatively,  
one can use the mixed Hodge structure on the local 
cohomology of~\eqref{localcover} 
supported at the singularity of~\eqref{localcover}.
Dualizing one obtains the mixed Hodge structure on the homology as well.

\begin{proof} 
Consider $z$ as a holomorphic 
function on the germ $V_{N,f}$ of the surface singularity 
$z^N=f(x,y)$ (cf. (\ref{localcover})) and its Milnor fiber $F_{N,f}$
i.e. the subset of $V_{N,f}$
given by the equation $z=t$ for fixed $t$. It has canonical identification 
with the Milnor fiber of $f(x,y)$ over $t^N$. Monodromy of the Milnor 
fiber of $z$ coincides with $T^N$ where $T$ is the monodromy of $f=t$.
Denote by $Z_0$ the subset of $V_{N,f}$ given by $z=0$ and 
consider the Wang exact sequence of the mixed Hodge structures 
(cf. for example 
\cite{dimcasaito} section 1.7)):
\begin{equation}
  H^1(V_{N,f}-Z_0) 
\rightarrow  H^1(F_{N,f}) \buildrel {T^N-I} \over 
\longrightarrow H^1(F_{N,f})(-1) \rightarrow 
H^2(V_{N,f}-Z_0) \rightarrow 0
\end{equation} 
The last term of this sequence is zero since the Milnor fiber has 
the homotopy type of a 1-complex. Since $T^N=I$ and since 
$L_{N,f}^{\circ}=L_{N,f}-Z_0$ has the mixed Hodge structure 
on its cohomology is 
constructed via the identification with the mixed Hodge structure 
on $V_{N,z}-Z_0$ we obtain isomorphism:
\begin{equation}\label{partialisomorphism}
   H^1(F_{N,f})(-1)=H^2(L_{N,f}^{\circ})
\end{equation} 

It provides an identification of the cohomology of punctured neighborhood
$L_{f,N}^{\circ}$ with the cohomology of Milnor fiber
i.e. $$H^2(L_{f,N}^{\circ})=H^1(F)(-1)$$
Moreover,  the exact sequence of pair:
\begin{equation}
H^2(L_{N,f},L_{N,f}^{\circ}) \rightarrow 
H^2(L_{f,N}) \rightarrow H^2(L_{f,N}^{\circ}) \rightarrow H^3
(L_{N,f},L_{N,f}^{\circ})
\end{equation} 
yields the isomorphism 
\begin{equation}
 H^2(L_{f,N})=H^2(L_{f,N}^{\circ})
\end{equation}
This together with (\ref{partialisomorphism}) yields the claim of the lemma. 
\end{proof}

Construction of the mixed Hodge structure on either side of 
(\ref{lemmaisomorphism}) yields  
that the type of this mixed Hodge structure on corresponding {\it homology} 
is 
\begin{equation}
(-1,-1),(-1,0),(0,-1),(0,0)
\end{equation}
In fact with assumption of 
semi-simplicity made in lemma \ref{comparemhs} the mixed Hodge structure 
on the homology of Milnor fiber is pure of type $(-1,0),(0,-1)$.

The equivalence of categories of such mixed Hodge structures without torsion 
and the categories of 1-motifs constructed in~\cite[10.1.3]{deligneIII} allows to 
extract the abelian variety $A$ which is the part of structure of 1-motif.
We shall refer to this abelian variety as the local 
Albanese variety of the singularity~\eqref{planesing}. 
The polarization of the Hodge structure on $\Gr_{-1}^W$,
required for such equivalence is the standard polarization of such graded component 
associated with the Milnor fiber. 

In the case of the cusp $x^2=y^3$
the assumptions of the lemma \ref{comparemhs} are fulfilled.
One can also describe the construction of corresponding local 
Albanese variety as taking the quotient of
$(\Gr^0_F\Gr^W_1)^*$ by the homology lattice of the closed
Riemann surface which is the compactified Milnor
fiber. In particular  
the Hodge structure on the cohomology of its Milnor fiber $x^2=y^3+t$
is the Hodge structure of the elliptic curve $E_0$ with
$j=0$. 

The above discussion yields: 

\begin{cor}
The eigenvalues of the generator 
$z \rightarrow z\cdot \exp({{2 \pi \sqrt{-1}} \over N})$ 
of the group of covering transformations (\ref{localcover})
acting on the cohomology of the singularity link~\eqref{localcover}
coincide with the eigenvalues of the 
action of the monodromy on the cohomology of the Milnor
fiber~\eqref{planesing} (cf. \cite{Duke} or \cite[section 1.3.1]{charvar}).
The eigenvalues of the above generator 
of the group of deck transformations of the germ \ref{localcover})
acting on $\Gr^0_F\Gr^W_1(H^1(L_{N,f}))$ 
have the form $e^{2 \pi \sqrt{-1} \alpha}$ where $\alpha$ runs through 
the elements  the spectrum of the singularity~\eqref{planesing}
which belongs to the interval $(0,1)$ (cf.~\cite{loeser}).
\end{cor}
 
For example, 
in the case of the cusp $x^2=y^3$ the monodromy, corresponding to the 
path $e^{2 \pi \sqrt{-1} s}$ ($0 \le s \le 1$) in  
the positive direction given by the complex structure
yields the monodromy of the Milnor fiber $F_{N,f}$ in lemma
\ref{comparemhs} given by 
\begin{equation}
(x,y,z) \rightarrow (xe^{{2 \pi \sqrt{-1} s}\over 2},
ye^{{2 \pi \sqrt{-1} s}\over 3}
ze^{{2 \pi \sqrt{-1} s}\over 6})
\end{equation}

In order to apply Theorem~\ref{roanth} to $\Alb(V_6)$, 
we shall need the following:

\begin{lemma}
\label{albaneseeigenvalues}
Let, as above, $V_6$ be a smooth $\ZZ_6$-equivariant 
model of a cyclic 6-fold covering 
space of $\PP^2$ branched over a curve with only nodes and cusps as singularities.
Let $T$ be a generator of covering group of $V_6$. The automorphism of $\Alb(V_6)$ 
induced by $T$ has only one eigenvalue (which is a primitive root of unity of degree six).
\end{lemma}
 
\begin{proof} It follows from~\cite{Duke}
that the Alexander module of a cyclic 
multiple plane branched over a curve with nodes and cusps, up to 
summands $\QQ[t,t^{-1}]/(t-1)$, is isomorphic to a direct sum  
\begin{equation}\label{alexandersum}
[\QQ[t,t^{-1}]/(t^2-t+1)]^s.
\end{equation}
More precisely, the proof of the Divisibility Theorem~\cite{Duke} shows that 
the Alexander module is a quotient of the direct sum of 
the Alexander modules of all singularities or equivalently the 
direct sum of the homology 
of Milnor fibers of singularities with module structure given by 
the action of the monodromy (in the present case of the cusp $x^2=y^3$
the local Alexander module is just one summand in~\eqref{alexandersum}).

Since the divisibility result is stated in~\cite{Duke}
with assumption of irreducibility of the ramification locus $\cC$ 
of~\eqref{multipleplane} we shall review the argument.
Denote by $T(\cC)$ a tubular neighborhood of $\cC$
in $\PP^2$ and consider the surjection 
\begin{equation}\label{surjectivityone}
\pi_1(T(\cC)\setminus \cC) \rightarrow \pi_1(\PP^2\setminus \cC)
\end{equation}
induced by embedding (the surjectivity is a consequence 
of the surjectivity, 
for $D \subset T(\cC)$, of the map 
$\pi_1(D \setminus D \cap \cC) \rightarrow \pi_1(\PP^2 \setminus \cC)$,
in turn, following from the weak Lefschetz theorem).
The surjection (\ref{surjectivityone}) induces the surjection 
of the homology of 6-fold cyclic coverings: 
\begin{equation}
H_1((T(\cC)\setminus \cC)_6,\ZZ) \rightarrow H_1((\PP^2\setminus \cC)_6,\ZZ).
\end{equation}
On the other hand, the covering space $(T(\cC)\setminus \cC)_6$ decomposes as 
\begin{equation}
(T(\cC)\setminus \cC)_6=\bigcup_P (B_P\setminus \cC)_6 \cup (U\setminus \cC)_6,
\end{equation}
where $B_i$ are small regular neighborhoods 
of all singular points $P$ of $\cC$, $U$ is the regular 
neighborhood of the smooth locus of $\cC$ and the subscript 
designates the 6-fold cyclic cover. The corresponding Mayer Vietoris
sequence yields a surjection: 
\begin{equation}\label{mayervietoris}
\oplus_P H_1((B_P\setminus \cC)_6,\QQ) \oplus H_1((U\setminus \cC)_6,\QQ)
\rightarrow H_1((T(\cC)\setminus \cC)_6,\QQ).
\end{equation}
One can view the cohomology
of the 6-fold cover of $(B_P\setminus \cC)_6$ 
as the cohomology of the punctured neighborhood of 
the part of exceptional curve in resolution of 6-fold cover of $B_P$ 
branched over $B_P\cap \cC$ outside of proper preimage of $\cC$. 
For any cusp $P$ the deck transformation acting on
$\Gr^0_F\Gr^W_1(H_1((B_P\setminus \cC)_6))$ has as eigenvalue the same primitive
root of unity of degree six (corresponding to the element in the spectrum 
of $x^2=y^3$ in the interval $(0,1)$;
in the case of the action on cohomology this part contains 
only $5 \over 6$).

Each of the other spaces $(U\setminus \cC)_6, (T(\cC)\setminus \cC)_6$ 
appearing in~\eqref{mayervietoris}
can be viewed as a punctured neighborhood of a quasi-projective variety and 
as such also supports the canonical mixed Hodge structure
(cf.~\cite{durfee})
Moreover, the sequence~\eqref{mayervietoris} is a sequence 
of mixed Hodge structures. 
It is shown in~\cite{Duke} that the map 
\begin{equation}\label{inclusion}
H_1((T(\cC)\setminus \cC)_6,\QQ) \rightarrow H_1(V_6,\QQ),
\end{equation}
is surjective and that the composition of~\eqref{mayervietoris}
and~\eqref{inclusion} takes $H_1((U\setminus \cC)_6,\QQ)$ to zero.

Both sequences~\eqref{mayervietoris}
and~\eqref{inclusion} are equivariant with respect to the deck transformations.
The sequence~\eqref{inclusion} yields 
that any eigenvalues of $T$ acting on $\Gr^0_FH^1(V_6)$
must be an eigenvalue of $T$ acting on 
$\Gr^0_F\Gr_1^WH^1((T(\cC)\setminus \cC)_6,\QQ)$ which is different from 1
i.e. is the eigenvalue of monodromy acting on Milnor fiber of 
cusp and corresponding to the part of spectrum in $(0,1)$.

This implies that the action of the 
deck transformation on $\Alb(V_6)$ is the multiplication by the same
root of unity of degree $6$ as well (exponent of the only element 
of the spectrum belonging to $(0,1)$).
Also note that an $i$-th power of this automorphism ($1 \le i <6$)
has zero as the only fixed point, since the existence of a fixed point 
for such $i$ would yield an eigenspace of the monodromy corresponding 
to an eigenvalue which is not a primitive root of degree~6.
\end{proof}

Theorem~\ref{roanth}, when applied to $X=\Alb(V_6)$ with
$g=q=h^{1,0}(V_6)$, and $d=6$,
shows (since the condition on the fixed point sets follows from the explicit
form of the action on $H^0(V_6,\Omega^1_{V_6})$) that
$k=q$ and that each component
$X_1,...,X_q$ is the elliptic curve with an automorphism of order six 
i.e. the curve $E_0$ with $j=0$. 

Hence we obtain the following:

\begin{prop}\label{mapstoE}
Let $V_6$ be a 6-fold cyclic multiple plane with branching 
curve having only nodes and cusps as singularities with irregularity $q$. 
Then 
$$\Alb(V_6)=E_0^q.$$
In particular for the set of morphisms taking $P_0$ to the zero of $E_0$ one has:
$$\Mor(V_6,E_0)=\Hom(E_0^q,E_0).$$ 
\end{prop}

Next we shall reformulate this proposition in terms of the Mordell-Weil group of 
split elliptic threefold $W=V_6 \times E_0$ viewed as elliptic
curve defined over~$\CC(V_6)$ (and which can be viewed as a cover of $W_F$ 
cf. section \ref{pencilsonmultiple}).

Recall that given an extension $K/k$ and an abelian variety $A$ over $K$, one has an abelian variety 
$B$ over $k$ (called the \emph{Chow trace}) and homomorphism $\tau: B \rightarrow A$ defined over $k$ 
such that for any extension $E/k$ disjoint from $K$ and abelian variety $C$ over $E$ and 
$\alpha: C \rightarrow A$ over $KE$ exist $\alpha': C \rightarrow B$ such that $\alpha=\tau \circ \alpha'$
(cf.~\cite[p.~97]{Langneron}).

In particular, to $W$ over $\CC(V_6)$,
one can associate Chow trace which is the elliptic curve $B$ over $\CC$ 
such that 
quotient of the group of $\CC(V_6)$-points of $W$ by the 
subgroup of $\CC$-points of $B$ is a finitely generated abelian group.
(Mordell-Weil Theorem cf.~\cite{Manin} and~\cite[Theorem 1]{Langneron}).
The Chow trace $B$ of $W$ is $E_0$ and the 
group $\tau B_{\CC}$ is the subgroup in $\CC(V_6)$ 
of points the corresponding 
to constant maps $V_6 \rightarrow E_0$. We shall denote by $\MW(W)$
the quotient of the group of $\CC(V_6)$ points by the subgroup 
of torsion point and the points of Chow trace.
As already mentioned, $\MW(W)$ is a finitely generated abelian group. 
Proposition~\ref{mapstoE} can be reformulated in terms of the Mordell-Weil group as follows:

\begin{cor}\label{cor-mw}
The Mordell-Weil group of $\CC(V_6)$-points of $V_6 \times E_0$ is 
a free $\ZZ[\omega_6]$ module having rank $q$
(where $q$ is the irregularity of $V_6$).
\end{cor}

\begin{proof} Non zero elements of this Mordell Weil group are represented
by the classes of non-constant sections of $V_6 \times E_0 \rightarrow V_6$.
Those corresponds to the maps $V_6 \rightarrow E_0$ up to translation.
Hence the corollary follows from Propositions \ref{mapstoE}.

\end{proof}

Next we shall return to the threefold $W_F$ which is a smooth birational
model of  
(\ref{equationWF}). We shall view it as an elliptic curve
over field $K=\CC(x,y)$. It splits over the field 
$K(F^{1 \over 6})=\CC(V_6)$.

\subsection{Elliptic pencils on multiple planes and $\PP^2$-points
of elliptic threefolds}\label{pencilsonmultiple}
We want to have an explicit correspondence between the elliptic pencils on $V_6$ and 
$\PP^2$ points of $W_F$ 
\footnote{Following classical terminology, by elliptic (resp. rational)
pencil we mean a morphism onto an elliptic (resp. rational) curve. An orbifold 
morphism onto a rational orbifold curve of elliptic type
induces an elliptic pencil cf. Remark~\ref{rem-elliptic}}. 
This is used in Theorem~\ref{irregularityMW} 
to obtain the relation between the Mordell Weil group of $\PP^2$-points 
of $W_F$ 
and Mordell Weil group of $\CC(V_6)$-points of split threefold 
$W=V_6 \times E_0$
and is based 
on Lemma~\ref{equationforw} below. Before we state it we shall 
introduce several notations.
Compactifications $W_F$ of the threefold  (\ref{equationWF}) 
have several useful biregular models 
in weighted projective spaces and their products which we shall derive.
We can view $E_0$ as the curve given by equation 
$u^3+v^2=w^6$ in weighted projective plane $\PP(2,3,1)$.
Let $\bar E_Q$ be surface in $\PP(2,3,1) \times \PP^1$ given by 
\begin{equation}\label{EQeq}
A^6(u^3+v^2)=w^6B^6.
\end{equation}
Projections to the factors we shall denote 
$pr_{\PP(2,3,1)}^{\bar E_Q}$ and $pr_{\PP^1}^{\bar E_Q}$ resp.
Clearly, the map $(u,v,w,A,B) \rightarrow (A^2u,A^3v,Bw,A,B)$ takes 
$\bar E_Q \subset \PP(2,3,1) \times \PP^1$ to the surface in $\PP(2,3,1) \times \PP^1$ 
given by $u^2+v^3=w^6$ (no dependence on $(A,B)$),
which is isomorphic to $E_0 \times \PP^1$. The action of $\mu_6$ on $\bar E_Q$
corresponding to the standard action on $E_0$ and trivial action on $\PP^1$
is given by $(A,B) \rightarrow (A,\omega_6B)$.

We denote by $\bar V_6$ the biregular model of the cyclic cover of $\PP^2$ branched 
over $F=0$ and $z=0$, which is the surface in $\PP^2 \times \PP^1$ given by 
\begin{equation}\label{V6eq}
\bar V_6:=\{([x:y:z],[M:N])\in \PP^2 \times \PP^1 \mid z^{6k}M^6=N^6F(x,y,z)\}.
\end{equation}
 
The action of the deck transformation is given by $(M,N)
\rightarrow (M,\omega_6N)$. 
The projection on the first (resp. the second) factor will be denoted 
$pr_{\PP^2}^{\bar V_6}$ (resp.~$pr_{\PP^1}^{\bar V_6}$).

We shall consider the threefold
\begin{equation}\label{equationbarw} 
\bar W=\bar V_6 \times_{\PP^1} \bar E_Q 
\end{equation} 
with the fibered product taken relative to the maps
$pr_{\PP^1}^{\bar V_6}$ and $pr_{\PP^1}^{\bar E_Q}$ respectively
with coordinates of respective copies of $\PP^1$ identified using 
the relation:
\begin{equation}\label{identification}
{N \over M}={A \over B}.
\end{equation}
The birational equivalence between $\bar E_Q$ and $E_0 \times \PP^1$ yields the 
birational isomorphism: 
\begin{equation}
\bar W \rightarrow \bar V_6 \times E_0,
\end{equation}
i.e. $\bar W$ is a projective model of $W$.
\begin{lemma}\label{equationforw}
The threefold $\tilde W_F$ in $\PP(2,3,1,1,1)$ given by
\begin{equation}\label{formulaequationforw}
(u^3+v^2)z^{6(k-1)}=F(x,y,z)
\end{equation}
is birationally equivalent to $W/\mu_{6}$ with the
diagonal action of $\mu_6$. It contains the hypersurface
(\ref{equationWF}) 
as an open set i.e. is a model of $W_F$. 
\end{lemma}

\begin{proof} 
The equations of $\bar W$ in $\PP^2 \times \PP^1 \times \PP(2,3,1) \times \PP^1$ 
are (\ref{EQeq}), (\ref{V6eq}) and (\ref{identification}).
Hence $\bar W$ is biregular to complete intersection in 
$\PP^2 \times \PP(2,3,1) \times \PP^1$ given by
\begin{equation}\label{completeint}
z^{6k}M^6=N^6F(x,y,z) \ \ \ N^6(u^3+v^2)=w^6M^6.
\end{equation}
Projection of this complete intersection on $\PP^2 \times \PP(3,2,1)$ has 
as its image the set of points $(u,v,w,x,y,z)$ for which the determinant 
of the system (\ref{completeint}) in $M^6,N^6$
is zero i.e. is the hypersurface given by:
\begin{equation}\label{intermideq}
z^{6k}(u^3+v^2)=w^6F(x,y,z).
\end{equation}
Clearly this projection is a cyclic $\mu_6$-covering of the
hypersurface (\ref{intermideq}).
Alternatively it is the quotient of (\ref{completeint})
by the action of $\mu_6$ given by $(M,N)\rightarrow (M,\omega_6N)$. 
Moreover it shows that (\ref{intermideq}) is the quotient of $\bar W$ by 
the {\it diagonal} (as follows from (\ref{identification})) action on $\mu_6$.
Finally both, the hypersurface in the statement of Lemma~\ref{equationforw}
and (\ref{intermideq}) have $W_F$ as Zariski open subset which 
yields the statement.
\end{proof}

\bigskip

Next we compare the $\CC(V_6)$-points of $E_0$ and $\PP^2$ points of elliptic 
threefold $W_F$. For $V_6$, which is a non-singular model of $\bar V_6$, 
the $\CC$-split elliptic threefold $W=V_6 \times E_0$, as above, is 
the elliptic curve over $\CC(V_6)$ obtained by field extension $\CC(V_6)/\CC$.
$\CC(V_6)$-points of $W$ correspond to rational maps $V_6 \rightarrow E_0$
which also can be viewed as sections of the projection $W \rightarrow V_6$ 
by associating with a map its graph in $W$ and vice versa. The corresponding 
Mordell-Weil group was calculated in Corollary~\ref{cor-mw}. 

The group $\mu_6$ acts (diagonally) on $W$ and hence also on $\MW(W)$.
We denote the invariant subgroup as $\MW(W)^{\mu_6}$. 
To a $\mu_6$-invariant element $V_6 \rightarrow W$ of $\MW(W)$
corresponds $\mu_6$-invariant $V_6$-point $\phi: V_6 \rightarrow E_0$ in the 
sense that $\phi(\gamma (v))=\gamma \phi(v), (\gamma \in \mu_6)$.
Its graph $\Gamma_{\phi} \subset V_6 \times E_0$ is $\mu_6$-invariant and taking 
the $\mu_6$-quotient yields the map: 
\begin{equation}
\PP^2=V_6/{\mu_6}=\Gamma/\mu_6 \rightarrow W/\mu_6=W_F.
\end{equation}
Hence we obtain a $\PP^2$-point of $W_F$.
Vice versa a section $\PP^2 \rightarrow W_F$ lifts to a birational map
of cyclic covers i.e. the map $V_6 \rightarrow W=V_6 \times E_0$
(which follows from comparison of the complements 
to the branching loci of both coverings).
This yields an equivariant (i.e. commuting with $\mu_6$-action) 
 elliptic pencil.

\begin{theorem}\label{irregularityMW}
The correspondence $\phi \mapsto \Gamma_{\phi}/\mu_6$ induces an isomorphism 
$\MW(W)^{\mu_6} \rightarrow \MW(W_F)$. In particular $\rk \MW(W_F)=2q(V_6)$.
\end{theorem}

\begin{proof}
It is enough to check that for two equivariant maps $\phi_1,\phi_2$
in the same coset of the Chow trace the map $V \rightarrow E_0$ given by 
$v \rightarrow \phi_1(v)-\phi_2(v)$ is constant with image $0 \in E_0$. 
Indeed this 
is a map to a point which due to equivariance should be the $\mu_6$-fixed point 
of $E_0$ i.e. zero.
 
To see the second part, since $\End(E_0)=\ZZ[\omega_6]$, 
we infer from Proposition \ref{mapstoE} or Corollary \ref{cor-mw},
that 
\begin{equation} 
Mor(V_6,E_0)=Hom(E_0^q,E_0)=\ZZ[\omega_6]^q 
\end{equation} 
The action of the group $\mu_6$ on $\Alb(V_6)$ is via multiplication
by $\omega_6$ (i.e. as is in the case of local Albanese of the
cusp). Since all elements in $\End(E_0)$ commute with
complex multiplication the action of $\mu_6$ on $\End(E_0)$
is trivial and one obtains $2q$ as the $\ZZ$ rank of $\MW(W_F)$.
 
\end{proof}

\begin{proof} (of Theorem~\ref{theoremonWF}). It follows 
immediately from
Theorem~\ref{irregularityMW} and the well-known relation between 
the degree of the Alexander polynomial and the irregularity of cyclic multiple planes 
(cf.~\cite{Duke}).
\end{proof}

\begin{cor}(cf.~\cite{kloo})
For a degree 6 curve with 6,7,8, and 9 cusps, the $\ZZ$-ranks of $\MW(W_F)$
are equal to 0,2,4, and 6 respectively.
\end{cor}

\subsection{A bound of the rank of Mordell-Weil group of an elliptic threefold}
\begin{theorem}\label{boundcusps}
If $d=6k$ is the degree of a homogeneous polynomial $F\in \CC[x,y,z]$,
for which the corresponding curve 
$\cC:=\{F=0\}$ has only nodes and cusps as singularities then
the $\ZZ$-rank of the Mordell-Weil group of $W_F$ satisfies:
$$\rk \MW(W_F) \le \frac{5}{3}d-2.$$ 
\end{theorem}

\begin{proof} 
Let $\ell \subset \PP^2(x,y,z)$ be a generic line in the base of the elliptic threefold 
$\pi: W_F \rightarrow \PP^2$. 
Then $\MW(W_F) \rightarrow \MW(\pi^{-1}(\ell))$ is injective 
(cf.~\cite{kloo}) and a bound on 
$\rk\MW(\pi^{-1}(\ell))$ therefore yields a bound on rank of $\MW(W_F)$.

On the other hand:
$$h^{1,1}(\pi^{-1}(\ell)) \ge rk NS(\pi^{-1}(\ell)) \ge rk \MW(\pi^{-1}(\ell)).$$

The surface $\pi^{-1}(\ell)$ is a hypersurface in the weighted
projective space $\PP^3(3k, 2k,1,1)$, which is a quotient 
of the surface in $\PP^3$ given by the equation:
$${\cal W}: \ \ P^{2}=Q^{3}+F\vert_{\ell}$$ 
with the action of the group 
$$G=\ZZ_{3k} \oplus \ZZ_{2k} \ \ \ (P,Q,a,b) \rightarrow (\omega_{3k}^i P,
\omega_{2k}^j Q, a,b)$$
($a,b$ are the coordinates of $\ell$).
Since the fixed points are outside of $\cal W$ the surface $\pi^{-1}(\ell)$ 
is non singular. 

In particular the elliptic surface 
$\pi^{-1}(\ell)$ gives an elliptic fibration with $6k=\deg F$ each degenerate fiber
being isomorphic to a cubic curve with one cusp. Hence 
$$e_{\text{top}}(\pi^{-1}(\ell))=12k$$ i.e. by Noether formula: 
$$\chi(\pi^{-1}(\ell))=k={d \over 6}$$
Hence, using the calculation of the Hodge diamond of an elliptic 
surface (cf.~\cite[\S6.9]{shioda}) and Tate-Shioda formula 
(cf.~\cite[Cor. 6.13]{shioda}) we derive: 
\begin{equation}
\rk \MW(\pi^{-1}(\ell))=\rho(\pi^{-1}(\ell))-2\le h^{1,1}-2=10\chi-2=10k-2,
\end{equation}
which gives the claim.
\end{proof}
 
This immediately yields:

\begin{cor}\label{inequality}
The degree of the Alexander polynomial of an irreducible curve $\cC$ of degree $d=6k$, 
whose singularities are only nodes and cusps satisfies:
$$\deg \Delta_\cC \le {5 \over 3}d-2.$$
\end{cor}

\section{Quasi-toric relations corresponding to cuspidal curves}
\label{sec-qt-rels}
\subsection{Quasi-toric relations and $\PP^2$-points of elliptic threefolds}

In this section we shall present an explicit relation between the quasi-toric 
relations introduced in section~\ref{quasitoricsection} 
and the elements of the Mordell Weil group of~$W_F$. Such a relation 
is expected since it was shown earlier that quasi-toric relations 
correspond to elliptic pencils on cyclic multiple planes and 
orbifold pencils (cf. section~\ref{quasitoricsection}). 
On the other hand in the past section such pencils were 
related to the Mordell-Weil groups.

Let us consider map of (\ref{formulaequationforw}) onto $\PP^2$ induced by 
the projection centered at $x=y=z=0$
\begin{equation}
\label{eq-proj}
\begin{matrix}
[u,v,x,y,z]\in & \PP(2,3,1,1,1)\setminus \{[u,v,0,0,0]\}\\
\downarrow & \downarrow\\
[x,y,z] & \PP^2
\end{matrix}
\end{equation}
Note that a rational section of this
projection is given by 
\begin{equation}
s(x,y,z)=[f(x,y,z),g(x,y,z),x\tilde h(x,y,z),y\tilde h(x,y,z),z\tilde h(x,y,z)]
\end{equation}
where
$2(\deg \tilde h+1)=\deg f$ and $3(\deg \tilde h+1)=\deg g$, satisfying 
$$(z\tilde h)^{6(d-1)}\left( {f^3}+{g^2}\right)=(\tilde h^d)^6F$$ and hence
$$z^{6(d-1)}\left(f^3+g^2\right)=\tilde h^6 F.$$
Since $F$ can be chosen not to be divisible by $z$, by the unique 
factorization property of $\CC[x,y,z]$, 
$h:=\frac{\tilde h}{z^{6(d-1)}}\in \CC[x,y,z]$, and hence
$$f^3+g^2=h^6 F$$
is a quasi-toric relation of $F$.

Conversely, given any quasi-toric relation $f^3+g^2= h^6F$, for certain 
$f,g,h\in \CC[x,y,z]$ such that $2(\deg h+d)=\deg f$ and $3(\deg h+d)=\deg g$, 
the map
$s(x,y,z)=[f(x,y,z),g(x,y,z),xz^{d-1}h(x,y,z),yz^{d-1}h(x,y,z),z^dh(x,y,z)]$
results in a section of the projection~(\ref{eq-proj}).

As a consequence, we obtain the following:

\begin{prop}
\label{prop-number-qt}
The degree of the Alexander polynomial is equal to the number of 
equivalence classes of 
quasi-toric 
relations which correspond to independent elements (over $\ZZ[\omega_6]$) of the 
Mordell-Weil group of $W_F$ given by equation~(\ref{formulaequationforw}).
\end{prop}

Also, using Theorem~\ref{irregularityMW} and projection~(\ref{eq-proj}) one can give formulas 
for the additive structure in $\MW(W_F)$. By Corollary~\ref{cor-mw}, 
it is enough just to give the action of $\ZZ[\omega_6]$ on the sections in $\MW(W_F)$ 
and the addition.
Consider sections $\gs_1:=(\frac{f_2}{h_1^2},\frac{f_3}{h_1^3},x,y,z)$ and 
$\gs_2:=(\frac{g_2}{h_2^2},\frac{g_3}{h_2^3},x,y,z)$ of $W_F^{\circ}$ 
and assume for 
simplicity that $h_1=h_2=1$. Then one has the following.

\begin{prop}
Under the above conditions,
\begin{equation}\label{eq-action}
\omega_6 \gs_1=\left(\omega_6 f_2,- f_3,x,y,z\right),
\end{equation}
also, if $\gs_1\neq \gs_2$
\begin{equation}\label{eq-sum}
\begin{matrix}
\gs_1+\gs_2=
\left(\frac{g_2f_2^2+g_2^2f_2+2g_3f_3-2F}{(f_2-g_2)^2},
\frac{3f_2g_2(f_3g_2-g_3f_2)+(f_3-g_3)(g_3f_3-3F)}{(f_2-g_2)^3},x,y,z\right).
\end{matrix}
\end{equation}
Otherwise
\begin{equation}\label{eq-double}
\begin{matrix}
2\gs_1=
\left(-f_2\frac{9f_2^3+8f_3^2}{4f_3^2},
-\frac{27f_2^6+36f_2^3f_3^2+8f_3^4}{8f_3^3},x,y,z\right).
\end{matrix} 
\end{equation} 
\end{prop}

\begin{proof}
Since $\omega_6 (u,v)=(\omega_6 u,-v)$, one
obtains~(\ref{eq-action}). The formulas can be easily obtained from the well-known formulas
of the group law in $E_0$ (cf.~\cite[Chapter III.3]{silverman}).
\end{proof}

Next we shall describe a geometric property of generators or the Mordell Weil 
group viewed as elliptic pencils.

\subsection{Primitive quasi-toric relations and orbifold maps}

An alternative way to see that every quasi-toric relation of $F$ contributes to the degree of
its Alexander polynomial comes from the theory of orbifold surfaces and orbifold morphisms 
reviewed in  section ~\ref{orbifoldcurves}. 

As mentioned before, the elements of the Mordell-Weil 
group $\MW(W_F)$ are represented by 
$\mu_6$-equivariant 
surjective maps $V_6 \rightarrow E_0$ (which we called elliptic pencils). 
To end this section, we will relate 
those which have irreducible generic members to generators of the 
Mordell-Weil group.

\begin{dfn}
We call an elliptic pencil $V_6 \rightarrow E_0$ \emph{primitive} if its generic fiber is irreducible.
\end{dfn}

One has the following result.

\begin{prop}\label{prop-primitive}
An elliptic pencil $f:V_6 \rightarrow E_0$ is non-primitive if and only if there exists an elliptic 
non-injective homomorphism $\sigma \in \Hom(E_0,E_0)$ such that $f=\sigma \circ \varphi$.
\end{prop}

\begin{proof}
It is enough to check that the Stein factorization of a non-primitive pencil
$f:V_6 \rightarrow E_0$ factorizes through an elliptic curve $f=\sigma \circ \varphi$,
where $\sigma\in \Hom(E_0,E_0)$ and $\varphi$ is a primitive elliptic pencil. 

In order to see this, one can use the cyclic order-six action $\mu_6$ and obtain a pencil
$\tilde f:V_6/\mu_6 \to E_0/\mu_6=\PP^1_{2,3,6}$, where $\PP^1_{2,3,6}$ is $\PP^1$
with an orbifold structure 2, 3, 6. Using the Stein factorization on $\tilde f$
one obtains $V_6/\mu_6\ \rightmap{\tilde \varphi}\ S\ \rightmap{\tilde \sigma}\ \PP^1_{2,3,6}$,
where $\tilde f=\tilde \sigma \circ \tilde \varphi$. Since $V_6/\mu_6$ is a rational 
surface, one obtains that $S=\PP^1$ with an orbifold structure given by at least three
orbifold points of orders 2,3, and 6. The orbifold points of order 2 and 3 correspond to double
and triple fibers, whereas the order 6 point corresponds to a non-reduced (but not multiple) 
fiber of type $h^6F$ where $F$ is the branching locus of the 6-fold cover of $\PP^2$.
By Proposition~\ref{prop-bound-pencil}, pencils of curves cannot have more than two 
multiple members, hence $S=\PP^1_{2,3,6}$. The 6-fold cover $E_0$ of $S$ ramified with orders 
2, 3, and 6 on the three orbifold points of $S$ allows for the existence of a factorization of 
$f$, say $V_6\ \rightmap{\varphi}\ E_0\ \rightmap{\sigma}\ E_0$ induced by $\tilde \varphi$ 
and 
$\tilde \sigma$. Since $\tilde \varphi$ is primitive, the map $\varphi$ is also primitive.
\end{proof}

\begin{remark}
Note that the elliptic pencil obtained by $2\gs_1$ should be non-primitive, since it 
factors through $E_0\to E_0$, given by the degree 4 map $x\mapsto 2x$ (see Proposition~\ref{prop-primitive}).
In fact, $2\gs_1$ produces the following quasi-toric relation
$$
k_2^3+k_3^2+64f_3^6F=0,
$$
where $k_2:=f_2(9f_2^3+20f_3^2)$, $k_3:=(27f_2^6+36f_2^3f_3^2+8f_3^4)$ (from (\ref{eq-double})), 
and $\alpha k_2^3+ \beta k_3^2=H_{(\alpha,\beta)}(f_2^3,f_3^2)$ is non-irreducible since
$H_{(\alpha,\beta)}(x,y)=
1728(\beta-\alpha)x^2y^2+(576\alpha-512\beta)xy^3+729(\beta-\alpha)x^4+64\beta y^4+1944(\beta-\alpha)x^3y$
which decomposes into a product of four factors of type $(y-\lambda(\alpha,\beta)x)$, since 
$H_{(\alpha,\beta)}(x,y)$ is a homogeneous polynomial of degree four in $x,y$.

Note that $k_2^3$ and $k_3^2$ are the union of four members (counted with multiplicity) of the pencil 
$\alpha f_2^3+\beta f_3^2$. Also, note that $\alpha k_2^3+\beta k_3^2$ is
a non-primitive pencil whose generic member also 
consists of four members of $\alpha f_2^3+\beta f_3^2$.

In particular, the 4:1 map $\tilde \sigma:\PP^1 \rightarrow \PP^1$, given by
$$
\tilde \sigma([x:y:z]):=[x(9x-20y)^3:(27x^2-36xy+8y^2)^2]
$$
is such that $\tilde f=\tilde \sigma \circ \tilde \varphi$ is the Stein factorization of $\tilde f$, 
where $\tilde \varphi([x:y:z])=[f_2^3(x,y,z):-f_3^2(x,y,z)]$ and
$\tilde f([x:y:z])=[k_2^3(x,y,z):k_3^2(x,y,z)]$.
\end{remark}

Summarizing the previous results, one obtains.

\begin{cor}
\label{cor-gen-primitive}
Let $2q$ be the degree of the Alexander polynomial of $F$, then there exist $q$ primitive quasi-toric 
relations $\gs_1,\dots,\gs_q$ of $F$ such that 
$\MW(W_F)=\ZZ[\omega_6]\gs_1 \oplus \dots \oplus \ZZ[\omega_6]\gs_q$,
where the action of $\ZZ[\omega_6]$ on $\gs_i$ is described 
in~(\ref{eq-action}), (\ref{eq-sum}), and (\ref{eq-double}).
\end{cor}

\begin{proof}
By Proposition~\ref{prop-number-qt} there exist $q$ quasi-toric relations $\gs_1,\dots,\gs_q$ of 
$F$ such that
$\MW(W_F)=\ZZ[\omega_6]\gs_1 \oplus \dots \oplus \ZZ[\omega_6]\gs_q$. The only question left to be proved is 
whether or not the quasi-toric relations $\gs_i$ can be chosen to be primitive. 
By Proposition~\ref{prop-primitive}, if $\gs_i$ was not primitive, then $\gs_i=\sigma \tilde \gs_i$ for a certain
$\sigma\in \ZZ[\omega_6]=\Hom(E_0,E_0)$ and $\tilde \gs_i$ primitive. Since $\sigma$ cannot be a unit,
then $\tilde \gs_i\notin \ZZ[\omega_6] \gs_1 \oplus \dots \oplus \ZZ[\omega_6]\gs_q$, which contradicts 
our assumption.
\end{proof}

\begin{theorem}
\label{cor-kulikov}
For any irreducible plane curve $\cC=\{F=0\}$ whose only singularities are nodes and cusps the
following statements are equivalent:
\begin{enumerate}
 \item\label{cor-kulikov-1} 
$\cC$ admits a quasi-toric relation of elliptic type $(2,3,6)$,
 \item\label{cor-kulikov-2} 
$\cC$ admits an infinite number of quasi-toric relations of elliptic type $(2,3,6)$, 
 \item\label{cor-kulikov-3} 
$\Delta_\cC(t)$ is not trivial ($\Delta_\cC(t)\neq 1$).
\end{enumerate}

Moreover, the set of quasi-toric relations of $\cC$
$\{(f,g,h)\in \CC[x,y,z]^3\mid f^2+g^3+h^6F=0\}$ has a group structure and it is 
isomorphic to $\ZZ^{2q}$, where $\Delta_\cC(t)=(t^2-t+1)^q$. Also, $\cC$ admits 
an infinite number of primitive quasi-toric relations, unless $q=1$,
in which case $\cC$ only has one primitive quasi-toric relation.
\end{theorem}

\begin{proof}
For the first part, \eqref{cor-kulikov-1} $\Leftrightarrow$ \eqref{cor-kulikov-2} 
is an immediate consequence of the group structure of the set of quasi-toric relations,
namely, once a quasi-toric relation $\gs$ is given, the set $\ZZ[\omega_6] \gs$ 
provides an infinite number of such relations.
Also~\eqref{cor-kulikov-1} $\Leftrightarrow$ \eqref{cor-kulikov-3} is a consequence
of Proposition~\ref{prop-number-qt}.
 
The \emph{moreover} part is a consequence of Proposition~\ref{prop-primitive}.
\end{proof}

Theorem~\ref{cor-kulikov} is related to~\cite{kulikov-albanese}. Also this result has been recently
noticed by Kawashima-Yoshizaki in~\cite[Proposition~3]{kaw-torus-type}. In this paper, a quasi-toric relation
$s$ is considered and then the series $\gs_n:=(\omega_6+1)^n s$ is used (it is an easy exercise).
The resulting quasi-toric relations $\gs_n$ are not primitive and the generic member of the pencil associated
with $\gs_n$ is the product of $3^n$ members of the original pencil associated with $s$,
which incidentally is the degree of $E_0\to E_0$, $x\mapsto (\omega_6+1)^nx$.

\section{Alexander polynomials of $\delta$-curves}
\label{sec-delta-curves}

In this section we extend the results of previous sections to 
singularities more general than nodes and cusps. 

Let us consider the general situation described in section~\ref{defalpol}, that is,
let us fix $F\in \CC[x,y,z]$, a homogeneous polynomial of degree $d$ which is not a power, whose set of zeroes 
in $\PP^2$ is the curve $\cC$. By~\eqref{definingepsilon} this is equivalent to fixing a projective plane curve
$\cC=\cC_1\cup ... \cup \cC_r$ and a list of multiplicities $(\e_1,...,\e_r)$ (such that $\gcd(e_i)=1$) or a 
surjection $\e_F:\pi_1(\PP^2\setminus \cC_0\cup \cC)\to \ZZ$ (where $\cC_0$ is a line at infinity transversal to 
$\cC$) such that $\e_F(\gamma_i)=\e_i$, $\gamma_i$ a meridian around $\cC_i$. Recall that $\e_i$ corresponds to 
the multiplicity of each irreducible component of $F$.

\subsection{Definition of $\delta$-curves and classification of 
$\delta$-essential singularities with respect to $\e$ with 
$\delta=3,4,6$} 
 
\begin{dfn}\label{deltapartial} 
Let $(\cC,P)$ be a germ of a singular point of $\cC$. We call 
$(\cC,P)$ a \emph{$\delta$-essential singularity}
(resp. \emph{$\delta$-co-prime singularity}) w.r.t. $\e$ 
if and only if the roots of $\Delta_{\cC,\e,P}(t)$ are all 
$\delta$-roots of unity (resp. no root of $\Delta_{\cC,\e,P}(t)$ 
is a $\delta$-root of unity except for $t=1$). 
 
We say that a curve $\cC \subset \CC^2$ \emph{has only $\delta$-essential singularities} 
if there exists an epimorphism
$\e: H_1(\CC^2\setminus \cC)\to \ZZ$ 
such that $(\cC,P)$ is a $\delta$-essential singularity w.r.t. $\e_P$ for all 
$P\in \Sing(\cC)$ (see~\eqref{eq-defep} to recall the construction of $\e_P$). 
 
A curve $\cC$ is called \emph{$\delta$-partial} w.r.t.
a homomorphism $\e$ if
any singularity $P$ of $\cC$, is either $\delta$-essential or $\delta$-co-prime. 
 
We also call a curve \emph{$\delta$-total} w.r.t.
to a homomorphism $\e$ 
 if it is $\delta$-partial and all the roots of the global Alexander 
polynomial $\Delta_{\cC,\e}(t)$ w.r.t. $\e$ are roots of unity
of degree $\delta$ (not necessarily primitive).
\end{dfn}

\begin{remark}
\label{rem-6ess}
The curves whose only singularities are (reduced) nodes and cusps 
necessarily have $6$-essential singularities. 
As another example one can consider
$\cC=\cC_0\cup \cC_1$, where $\cC_1$ is the 
tricuspidal quartic and $\cC_0$ is its bitangent and the
epimorphism $\e$ mapping the meridian of $\cC_0$ to 
$2$ and the meridian of $\cC_1$ to $1$ (i.e. the homomorphism of type
$(2,1)$. The local 
Alexander polynomial of $\cC$ at a tacnode w.r.t. $\e$
is given by $(t^3+1)$, whereas at a cusp it is simply $t^2-t+1$.
Therefore all singularities of $\cC$ w.r.t. 
homomorphism of type $(2,1)$. are 6-essential.
In particular, by Proposition~\ref{prop-divisibility}, $\cC$ is a
$6$-total curve.

Also note that tacnodes with respect to the homomorphism of type
$(1,1)$ are 4-essential singularities as well as nodes with respect 
to the homomorphism of type~$(1,2)$.

As for 3-essential singularities, 
one has $\AAA_5$-singularities w.r.t. $\e$ of type 
$(1,1)$ and nodes w.r.t. $\e$ of type~$(1,3)$. 
\footnote{here and below we use standard ADE-notations for germs
of simple plane curve singularities. In particular, the germs of
$\AAA_n$-singularities (resp. $\DD_n$-singularities) 
are locally equivalent to $x^{n+1}+y^2$ (resp. $x^2y+y^{n-1}$).}. 
\end{remark}

\begin{prop}
\label{prop-essential-local-global}
A curve with only $\delta$-essential singularities is $\delta'$-total
for some~$\delta'|\delta$.
\end{prop}

\begin{proof}
The result is an immediate consequence of Proposition~\ref{prop-divisibility}.
\end{proof}

Note that the converse of Proposition~\ref{prop-essential-local-global} is 
not true (see section~\ref{sec-examples}).

Also, Proposition~\ref{prop-essential-local-global} can be sharpened 
using~\cite{degtyarev-divisibility}, so that not
all singularities are required to be $\delta$-essential.
 
\begin{remark}
\label{rem-total-partial}
Note that there are examples of $\delta$-partial curves that are not $\delta$-total,
for instance a union of two cuspidal cubics intersecting each other at three smooth
points with multiplicity of intersection 3, that is, a sextic $\cC$ with singularities
$2\AAA_2+3\AAA_5$
\footnote{i.e. the set of singularities consists of 2 points of type
$\AAA_2$ and three points of type $\AAA_5$. Similar notations will be used in the rest
of the paper.}
has Alexander polynomial $\Delta_\cC(t)=(t^2-t+1)^2(t^2+t+1)$
(cf.~\cite{oka-alexander}). Therefore $\cC$ is 3-partial but not 3-total. This curve
will be studied in more detail in Example~\ref{ex-2a2-3a5}.
\end{remark}

\begin{prop} 
\label{prop-delta-alexander}
Let $(\cC,P)$ be a germ of curve singularity and let 
$\pi: \tilde \CC^2 \rightarrow \CC^2$ be a birational 
morphism such that the support of the total pull-back
$\pi^*(\cC)$ is a normal crossing divisor on 
$\tilde \CC^2$. Let $m_{i,j}$ be the multiplicity of $\pi^*{f_j}$ 
where $f_j$ is the local equation of the $j$-th branch of $\cC$ at 
$P$ along the exceptional curve $E_i$ of $\pi$. Let $\e_j$ 
be the value of $\e$ on the meridian corresponding to $f_j$.
Then the local Alexander polynomial of $\cC$ at $P$ w.r.t. $\e$
is given by:
\begin{equation}
\label{eq-alexander-delta}
\Delta_{\cC,\e,P}=
(t-1)\Prod_i(1-t^{\sum_j m_{i,j}\e_j})^{-\chi(E_i^0)}
\end{equation}
where $E_i^0=E_i\setminus \bigcup_{k\mid k\ne i} E_k$ and $\chi$ is 
the topological Euler characteristic.
\end{prop}

\begin{proof} Indeed, the infinite cyclic cover of the complement 
to the zero set of the germ of $\cC$ at $P$ corresponding to the 
homomorphism $\e$ can be identified with the 
Milnor fiber of non-reduced singularity:
\begin{equation}
f_1^{\e_1} \cdot ....\cdot f_r^{\e_r}.
\end{equation}
Now the claim follows from A'Campo's Formula (cf.~\cite{acampo}).
\end{proof}

\begin{remark} 
Proposition~\ref{prop-delta-alexander} allows one to compile a complete list 
of $\delta$-essential singularities with $\delta \le k$.
\end{remark}

In Table~\ref{tab:d-ess} a list of the possible 6-essential singularities
is given. The first column shows the number of local branches of the 
singularity. The second column contains the reduced type of the singularity.
The third column shows a list of all possible multiplicities 
for the branches. It is a consequence of   
the divisibility conditions on the multiplicities of each branch
imposed by 
Proposition~\ref{prop-delta-alexander} which follows from the requirement 
to have a 6-essential singularity: see the proof below. 

The forth column gives the list of multiplicities $(s_1,s_2,s_3,s_6)$ of the 
irreducible factors of the Alexander polynomial in~$\QQ[t]$
$$
\Delta_{P,\e}(t)=(t-1)^{s_1}(t+1)^{s_2}(t^2+t+1)^{s_3}(t^2-t+1)^{s_6}.
$$

Finally the fifth column shows that the set of logarithms 
belonging to the interval $(0,1)$ of the eigenvalues
of (the semi-simple part of) the monodromy acting on $\Gr^0_F\Gr^W_1$
of the cohomology of the Milnor fiber. 

\begin{center}
\begin{table}[h]
\caption{6-essential singularities} 
\centering 
{\small{
\begin{tabular}{|c|c|c|c|c|} 
\hline 
$r$ & $\begin{matrix}\text{reduced} \\ \text{singularity type}\end{matrix}$
& possible types of $\e$ & $\begin{matrix}\text{Alexander} \\ \text{polynomials}\\ (s_1,s_2,s_3,s_6)\end{matrix}$
& $\begin{matrix}\text{weight 1 part}\\ \text{of spectrum} \\ \text{in } (0,1)\end{matrix}$
\\ 
\hline \hline
$1$ & $\AAA_2\equiv y^2-x^3$ 
& $(1)$ & $(0,0,0,1)$ 
& \raisebox{-.25ex}{$\frac{5}{6}$}
\\ [0.5ex]
\hline

\multirow{3}{*}{$2$} 
& $\AAA_1 \equiv y^2-x^2$ 
& $(1,1)$ & $(1,0,0,0)$
& \raisebox{-.25ex}{$\emptyset$}
\\ [0.5ex]
\cline{2-5}
& $\AAA_3\equiv y^2-x^4$ 
& $(2,1)$ & $(2,0,0,1)$
& \raisebox{-.25ex}{$\frac{5}{6}$}
\\ [.5ex]
\cline{2-5}
& $\AAA_5\equiv y^2-x^6$ 
& $(1,1)$ & $(1,0,1,1)$
& \raisebox{-.25ex}{$\frac{2}{3},\frac{5}{6}$}
\\ [.5ex]
\hline

\multirow{3}{*}{$3$} 
& $\DD_4\equiv y^3-x^3$ 
& 
${\left\{\begin{matrix}(4,1,1),(3,2,1),(2,2,2)\\(1,1,1)\end{matrix}\right.}$ &  
$\left\{\begin{matrix}(2,1,1,1)\\(2,0,1,0)\end{matrix}\right.$
& \raisebox{-.25ex}
{$\left\{\begin{matrix} \frac{2}{3},\frac{5}{6}\\ \frac{2}{3} \end{matrix}\right.$}
\\ [.5ex]
\cline{2-5}
& $\DD_6\equiv (x^2-y^4)y$ 
& $(1,1,2)$ & $(2,1,1,1)$
& \raisebox{-.25ex}{$\frac{2}{3},\frac{5}{6}$}
\\ [.5ex]
\cline{2-5}
& $(x^3-y^6)$ 
& $(1,1,1)$ & $(2,2,1,2)$
& \raisebox{-.25ex}{$\frac{2}{3},\frac{5}{6}$}
\\ [.5ex]
\hline

\multirow{2}{*}{$4$} 
& $(x^4-y^4)$ 
& $(3,1,1,1), (2,2,1,1)$ & $(3,2,2,2)$
& \raisebox{-.25ex}{$\frac{1}{3},\frac{2}{3},\frac{5}{6}$}
\\ [.5ex] 
\cline{2-5} 
& $(x^2-y^4)(y^2-x^4)$ 
& $(1,1,1,1)$ & $(3,2,2,2)$
& \raisebox{-.25ex}{$\frac{1}{3},\frac{2}{3},\frac{5}{6}$}
\\ [.5ex]
\hline

$5$ & $x^5-y^5$ 
& $(2,1,1,1,1)$ & $(4,3,3,3)$
& \raisebox{-.25ex}{$\frac{1}{3},\frac{2}{3},\frac{5}{6}$}
\\ [.5ex]
\hline

$6$ & $x^6-y^6$ 
& $(1,1,1,1,1,1)$ & $(5,4,4,4)$
& \raisebox{-.25ex}{$\frac{1}{3},\frac{2}{3},\frac{5}{6}$}
\\ [.5ex]
\hline
\end{tabular}
}}
\label{tab:d-ess}
\end{table}
\end{center}

\begin{prop}\label{proofoftable1}
Table~\ref{tab:d-ess} contains a complete list of $6$-essential singularities.
\end{prop}

\begin{proof}
We will use notation from Proposition~\ref{prop-delta-alexander}.
The contributing exceptional divisors (i.e. those with $\chi(E^{\circ})\ne 0$) 
appearing in the end of resolution of singularity $(\cC,P)$ have
maximal multiplicity $\sum m_{i,j}\e_j$ among components preceding it. 
Hence the primitive root corresponding to the
component with maximal such multiplicity won't cancel
in \eqref{eq-alexander-delta}. 
Moreover, the multiplicities of the 
contributing exceptional divisors 
have to divide~$6$. 

In the case $r=1$ this forces the singularity 
to be a cusp. Also, for $r>1$ this forces the irreducible branches to be smooth.
The rest of the list can be worked out just keeping in mind that since $\sum_i m_{i,j}\leq 6$,
in particular removing any branch from a valid singularity with $r+1$ branches, one 
should obtain a valid singularity with $r$ branches. Finally, if $\Sigma_{r}$ is
a $\delta$-essential singularity type of $r$ branches whose only valid homomorphism
is $(1,\dots,1)$, then there is no $\delta$-essential singularity $\Sigma_{r+1}$
of $(r+1)$ branches such that $\Sigma_r$ results from removing a branch from 
$\Sigma_{r+1}$.

For example, for reduced 
singularity $\AAA_3$ the 
Proposition~\ref{prop-delta-alexander} yields that a collection  
of multiplicities of branches $(\e_1,\e_2)$ 
yields a 6-essential non-reduced singularity 
if and only if $2(\e_1+e_2) \vert 6$. This is satisfied
only by the pair $(\e_1,\e_2)=(2,1)$.
Similarly, for singularity $\DD_4$ the divisibility condition 
is $\e_1+e_2+e_3 \vert 6$ which is satisfied by three triplets 
indicated in the table, etc.

Note that the possible morphisms $\e$ shown in the third column are all up to action
of the permutation group except for the case $(x^2-y^4)y$, where the last branch
is not interchangeable with the others. So, whereas $(3,2,1)$ in the $\DD_4$ case
represents six possible morphisms, $(1,1,2)$ in the $(x^2-y^4)y$ case only represents
one possible morphism.

In order to check the last column, it is enough to compute the constants of 
quasi-adjunction (cf.~\cite{arcataquasiadj}) different from $\frac{1}{2}$, since $k$ 
is a constant of quasi-adjunction different from $\frac{1}{2}$ if and only if $1-k$ is 
an element of the spectrum of the singularity corresponding to the part of weight 1 (cf.~\cite{loeser}).

Constants of quasi-adjunction can be found in terms of a resolution of the singularity as follows.

If $f_1^{a_1}(x,y)...f_r^{a_r}(x,y)$ is an equation of the germ of plane curve singularity 
at the origin, $\pi: V \rightarrow \CC^2$ is its resolution (i.e. the proper preimage of the 
pull-back is a normal-crossing divisor), $E_k$ are the exceptional components of the resolution 
$\pi$, $N_k$ (resp. $c_k$, resp. $e_k(\phi)$) is the multiplicity along $E_k$ of 
$\pi^*(f_1^{a_1}....f_r^{a_r})$ (resp. $\pi^*(dx \wedge dy)$, resp. $\pi^*(\phi)$, where 
$\phi(x,y)$ is a germ of a function at the origin) then for a fixed $\phi$, $\kappa_{\phi}$ is 
the minimal solution to the system of inequalities (the minimum is taken over $k$):

\begin{equation}
N_k\kappa_{\phi}\ge N_k-e_k(\phi)-c_k-1
\end{equation}
(cf. \cite[(2.3.8)]{charvar}).

For example, for $\AAA_2$-singularity (i.e. the cusp which can be resolved by three blow ups) 
one has for the last exceptional component:

$$e(x)=2, e(y)=3, N=e(x^2-y^3)=6, c=4$$
and hence $\kappa=\frac{1}{6}$. Therefore for the spectrum (i.e. the last column in 
Table~\ref{tab:d-ess}) one obtains the well-known constant $\frac{5}{6}$.

Similarly, for $\DD_4$-singularity $x^3-y^3=0$ (which is resolved by the single blow up) with 
the multiplicities of the components $(4,1,1)$ one obtains only one inequality 

\begin{equation}
6\kappa \ge 6-e_1(\phi)-2.
\end{equation}
Hence for various choices of $\phi$ one obtains $\kappa={1 \over 2}, {1\over 3}, {1\over 6}$.
Thus $\frac{2}{3}, \frac{5}{6}$ are the only elements of the spectrum in $(0,1)$ corresponding
to the part of weight 1.
\end{proof}

In the following tables the notations are the same as in the Table~\ref{tab:d-ess}
and in Table~\ref{tab:d-4-ess} the factorization of the Alexander polynomial is the 
following:
$$
\Delta_{P,\e}(t)=(t-1)^{s_1}(t+1)^{s_2}(t^2+1)^{s_4}.
$$

\begin{table}[h]
\caption{4-essential singularities} 
\centering 
{\small{
\begin{tabular}{|c|c|c|c|c|} 
\hline 
$r$ & $\begin{matrix}\text{reduced} \\ \text{singularity type}\end{matrix}$
& possible types of $\e$ & $\begin{matrix}\text{Alexander} \\ \text{polynomials}\\ (s_1,s_2,s_4)\end{matrix}$
& $\begin{matrix}\text{weight 1 part}\\ \text{of spectrum} \\ \text{in } (0,1)\end{matrix}$
\\ [0.5ex]
\hline \hline

\multirow{2}{*}{$2$} 
& $\AAA_1 \equiv y^2-x^2$ 
& $(\e_1,\e_2)$ & $(1,0,0)$
& \raisebox{-.25ex}{$\emptyset$}
\\ [0.5ex]
\cline{2-5}
& $\AAA_3\equiv y^2-x^4$ 
& $(1,1)$ & $(1,0,1)$
& \raisebox{-.25ex}{$\frac{3}{4}$}
\\ [0.5ex]
\hline

\multirow{1}{*}{$3$} 
& $\DD_4\equiv y^3-x^3$ 
& $(2,1,1)$ & $(2,1,1)$
& \raisebox{-.25ex}{$\frac{3}{4}$}
\\ [0.5ex]
\hline

\multirow{1}{*}{$4$} 
& $(x^4-y^4)$ 
& $(1,1,1,1)$ & $(3,2,2)$
& \raisebox{-.25ex}{$\frac{3}{4}$}
\\ [0.5ex]
\hline
\end{tabular}
}}
\label{tab:d-4-ess}
\end{table}

\begin{table}[h]
\caption{3-essential singularities} 
\centering 
{\small{
\begin{tabular}{|c|c|c|c|c|} 
\hline 
$r$ & $\begin{matrix}\text{reduced} \\ \text{singularity type}\end{matrix}$
& possible types of $\e$ & $\begin{matrix}\text{Alexander} \\ \text{polynomials}\\ (s_1,s_3)\end{matrix}$
& $\begin{matrix}\text{weight 1 part}\\ \text{of spectrum} \\ \text{in } (0,1)\end{matrix}$
\\ [0.5ex]
\hline \hline

\multirow{1}{*}{$2$} 
& $\AAA_1 \equiv y^2-x^2$ 
& $(\e_1,\e_2)$ & $(1,0)$
& \raisebox{-.25ex}{$\emptyset$}
\\ [0.5ex]
\hline

\multirow{1}{*}{$3$} 
& $\DD_4\equiv y^3-x^3$ 
& $(1,1,1)$ & $(2,1)$
& \raisebox{-.25ex}{$\frac{2}{3}$}
\\ [0.5ex]
\hline

\end{tabular}
}}
\label{tab:d-3-ess}
\end{table}

\begin{prop}
Table~\ref{tab:d-4-ess} (resp.~\ref{tab:d-3-ess}) classifies the 4-essential 
(resp. 3-essential) singularities.
\end{prop}
 
\noindent The proof is same as the proof of Proposition~\ref{proofoftable1}.

\begin{cor}\label{oneroot} 
For simple $\delta$-essential singularities (i.e. of type $\AAA$ or $\DD$) 
there is at most one eigenvalue of order $\delta' \vert \delta$ ($\delta' > 2$) 
for the action of the monodromy on the Hodge component $\Gr^0_F\Gr^W_1$
of the Milnor fiber. 
\end{cor}
 
\begin{proof} Inspection on Tables~\ref{tab:d-ess}-\ref{tab:d-3-ess} shows that 
except for $t=1$, the multiplicity of each root of the characteristic polynomial
of the monodromy is equal to one. Since the conjugation of an eigenvalue of the monodromy 
acting on $\Gr^0_F\Gr^W_1$ is an eigenvalue of its action on $\Gr^1_F\Gr^W_1$ it follows that 
$t=-1$ is not an eigenvalue of the monodromy action on  $\Gr^0_F\Gr^W_1$.
Again, by direct inspection of Tables~\ref{tab:d-ess}-\ref{tab:d-3-ess} the result follows.
\end{proof} 
 
\subsection{Decomposition of the Albanese variety of cyclic covers 
branched over $\delta$-total curves}

The following auxiliary results will be useful in the rest of the arguments. 
 
\begin{lemma}\label{deltaessentialalexander} 
Let $\cC$ be a $\delta$-partial curve with $r$ irreducible components. Then the polynomial
\begin{equation}\label{delta3}
\array{ll}
\Delta_{\cC,\e}^\delta:=(t-1)^{r-1}
\displaystyle\prod_{k|\delta, k>1} \varphi_k(t)^{s_k},
\endarray
\end{equation}
divides the Alexander polynomial $\Delta_{\cC,\e}$ of $\cC$ w.r.t.~$\e$,
where $s_i$ above denotes the multiplicity 
of the primitive root of unity of degree $i$ in 
the Alexander polynomial of~$\cC$.
 
Moreover, let $V_{\delta}$ be the cyclic cover of $\PP^2$ of degree $\delta$ 
branched over a curve $\cC$ according to the multiplicities 
$\e=(\e_1,\dots,\e_r)$. Then, the characteristic polynomial of the deck 
transformation acting on $H_1(V_{\delta},\CC)$ equals 
\begin{equation} 
{\Delta^\delta_{\cC,\e} \over {(t-1)^{r-1}}}. 
\end{equation} 
 
If $\cC$ is $\delta$-total, then 
$$\Delta^\delta_{\cC,\e}=\Delta_{\cC,\e}.$$
\end{lemma}

\begin{proof}The first (resp. last) assertion about the Alexander polynomial follows directly 
from the definitions of $\delta$-partial (resp. $\delta$-total) curves and the well-known 
fact that the multiplicity of the root $t=1$ for a curve with $r$ irreducible 
components in the Alexander polynomial is equal to $r-1$.

The \emph{moreover} part follows from the relation between the homology of 
branched and unbranched covers (cf.~\cite{Duke}).
\end{proof}

\begin{lemma}\label{localdecomposition} 
For a plane curve singularity $P$,
denote by $\Alb_P$ the local Jacobian (cf. section~\ref{subsec-MW}) 
Let $\cC$ be a $\delta$-partial curve and 
let $V_{\delta}$ be cyclic cover of degree $\delta$ of $\PP^2$ branched over $\cC$. 
Then $\Alb(V_{\delta})$ is isogenous to a quotient of the product of local Jacobians 
of singularities of~$\cC$. 
\end{lemma}

\begin{proof} 
It follows from the proof of Lemma~\ref{albaneseeigenvalues}.
\end{proof}
 
Finally, we will give a description of the multiplicities $s_i$ in terms
of the Albanese variety of the ramified coverings of $\PP^2$. Denote by 
$E_0$ (resp. $E_{1728}$) the elliptic curve with $j$-invariant zero 
(resp. $1728$). Then one has the following.
 
\begin{theorem} 
\label{thm-main-alb}
Let $\cC$ be a $\delta$-partial curve ($\delta=3,4,$ or $6$)
with $\delta$-essential singularities of type $\AAA$ and $\DD$. 
Then the Albanese variety $\Alb(V_{\delta})$ corresponding 
to the curve $\cC$ can be decomposed as follows: 
\begin{equation}\label{isogeniesdeltaalbanese}
\array{ll}
\Alb(V_{\delta}) = E_{0}^{s_3} & \text{ if } \delta=3,\\
\Alb(V_{\delta}) \sim A \times E_{1728}^{s_4} & \text{ if } \delta=4,\\
\Alb(V_{\delta}) \sim A \times E_{0}^{s_3+s_6} & \text{ if } \delta=6,\\
\endarray
\end{equation}
where $s_i$ is the multiplicity of the $i$-th primitive root of unity in 
$\Delta_{\cC,\e}(t)$, $A$ is an abelian variety of dimension $s_2$, 
``$=$'' means isomorphic, and ``$\sim$'' means \emph{isogenous}.
\end{theorem}

\begin{proof} 
The deck transformation of $V_{\delta}$ induces the action of the cyclic group 
$<\alpha>$ on the Albanese variety of $V_{\delta}$. It follows from Roan's 
Decomposition Theorem (cf.~\cite[Theorem 2.1]{Lange-Roan})
that $\Alb(V_{\delta})$ is isogenous to a 
product $X_1 \times .... \times X_e$ 
where $e$ is the number of orders of eigenvalues of $\alpha$. 
We shall show that abelian varieties $X_i$ can be decomposed 
further to yields isogenies \ref{isogeniesdeltaalbanese}.
By Lemma~\ref{deltaessentialalexander}, one has that $1$ is not an eigenvalue 
of $\alpha$, and hence each component $X_i$ supports an automorphism whose order 
is a non-trivial divisor of $\delta$.
We denote by $A$ the component supporting an automorphism of order 2.
In the cases of $\AAA$ or $\DD$ singularities, Corollary~\ref{oneroot} 
yields that the action of $\alpha$ on each component $X_i$ has 
at most one eigenvalue. Indeed this is the eigenvalue of the monodromy 
on the component $\Gr^0_F\Gr^W_1$ of the Milnor fiber of the singularities whose 
local Alexander polynomial contains such eigenvalue as a root.
Each eigenvalue of the monodromy on $\Gr^0_F\Gr^W_1$ for a $\delta$-essential 
singularity does not appear on $\Gr^1_F\Gr^W_1$ even if for different
$\delta$-essential singularities. In particular even different 
$\delta$-essential singularities contribute at most to the 
same root of the global Alexander polynomial as follows from the 
table in Corollary~\ref{oneroot}.
This allows one to apply the second Decomposition 
Theorem (cf.~\cite[Theorem 3.2]{Lange-Roan}) as was done above in the case of 
cuspidal curves. The result follows.
\end{proof}

\begin{remark} 
\mbox{}
\begin{enumerate}
\item
If $\cC$ has singularities of type $x^3-y^6$ then the above argument shows that 
$\Alb(V_6)$ has, up to isogeny, the factors which are the Jacobians 
of the projection model of $y^3=x^6-1$ and the elliptic curves of $E_{1728}$
and~$E_0$.

\item
Whenever $r > 3$, the local Jacobians (cf. section \ref{subsec-MW}) 
depend on the moduli describing the local algebraic analytic type
of singularities. For example, in the case 
of singularities $x^4-y^4$ the local Jacobian is the Jacobian 
of the 4-fold cyclic cover branched over 4 points. The quotient by 
an involution of such cover may yield an arbitrary elliptic curve in 
the Legendre family of elliptic curves. Hence, the Albanese variety 
of the cyclic cover $V_{\delta}$ is 
isogenous to the quotient of abelian varieties whose moduli 
a priori depends on the analytic type of the singularities.
It would be interesting to see if one can find examples 
showing that such variations can take place. 
\end{enumerate}
\end{remark}

\subsection{Elliptic threefolds corresponding to $\delta$-curves}

Now we shall relate the Mordell-Weil group of threefolds corresponding to 
$\cC$ to $\Delta_{\cC,\e}$.
Let $W$ denote an elliptic threefold birational to the affine
hypersurface given by the equation:
\begin{equation}\label{threefold3}
\array{ll}
u^2+v^3=F(x,y,1) & \text{ if } \delta=3,6,\\
u^2+v^3=F(x,y,1)v & \text{ if } \delta=4.
\endarray
\end{equation}
One has the following:

\begin{theorem} 
\label{thm-main-delta}
Let $\cC=\{F=0\}$ be a $\delta$-partial curve as in Theorem~\ref{thm-main-alb}, then
\begin{equation}
\array{ll}
\rk \MW(W_F) = s_3 & \text{ if } \delta=3,\\
\rk \MW(W_F) \geq s_4 & \text{ if } \delta=4,\\
\rk \MW(W_F) \geq s_3+s_6 & \text{ if } \delta=6.\\
\endarray
\end{equation}

In addition, if $\cC$ is a $\delta$-total curve, then 
\begin{equation}
\array{ll}
\rk \MW(W_F) = s_4 & \text{ if } \delta=4 \text{ and } \Delta_{\cC,\e}(-1) \ne 0,\\
\rk \MW(W_F) = s_3+s_6 & \text{ if } \delta=6 \text{ and } \Delta_{\cC,\e}(-1) \ne 0,\\
\endarray
\end{equation}
\end{theorem}

\begin{proof} As in the proof of Theorem~\ref{irregularityMW} one has 
the identification of $\MW(W_F)$ and $\mu_i$-invariant elements
of the Mordell-Weil group 
of $W_F$ over the extension $\mu_i(\CC(x,y))$ 
of $\CC(x,y)$ with the Galois group $\mu_i$,
where $i=6$ for $\delta=3,6$ and $i=4$ for $\delta=4$.
The threefold $u^2+v^3=F(x,y,1)v$ splits over the field 
$\CC(x,y)(F^{1 \over 4})$ since it is isomorphic to a
direct product threefold using $v'=vF^{1 \over 2}$, $u'=F^{3 \over 4}u$.
The inequalities follow from:
$\rk\MW(W_F,\CC(x,y)(F^{ 1\over {\delta}})) \ge \rk \Hom (E_k^s,E_k)$, 
where $k=0,s=s_3$ for $\delta=3$, $k=0,s=s_3+s_6$ for 
$\delta=6$ and $k=1728, j=s_4$ for $\delta=4$.
\end{proof}

\subsection{Statement and proof of the main theorem for $\delta$-curves}

\begin{theorem} 
\label{thm-bound-delta}
Let $\cC=\{F=0\}$ be a $\delta$-partial curve for $\delta=3,4,$ or $6$, then:
\begin{enumerate}
\item\label{thm-bound-delta-1}
There is a one-to-one correspondence between quasi-toric relations
corresponding to $F$ and the $\CC(x,y)$-points of the threefold $W_F$.
\end{enumerate}
If, in addition, the singularities of $\cC$ are as in Theorem~\ref{thm-main-alb} then:
\begin{enumerate}
\item[(2)]
The multiplicities of the factors of the global Alexander polynomial of
$\delta$-partial curves satisfy the following inequalities:
\begin{equation}
\array{ll}
s_3 \leq {5 \over 6} d-1 & \text{ if } \delta =3,\\
s_4 \leq {5 \over 6} d-1 & \text{ if } \delta =4,\\
s_3+s_6 \leq {5 \over 6} d-1 & \text{ if } \delta =6.\\
\endarray
\end{equation}
\end{enumerate}
\end{theorem}

\begin{proof} The argument is the same as in the proof of 
Theorem~\ref{boundcusps} since the rank of the Mordell-Weil group
of the threefolds for each $\delta$ is bounded by the 
rank of Mordell-Weil group of the elliptic surface with $d$ 
degenerate fibers each isomorphic to a cubic curve with a single cusp
as follows from equations (\ref{threefold3}). 
For each $\delta$ we obtain a bound on the rank of the submodule 
of the Alexander module corresponding to the action of the deck transformation 
on the subspaces generated by the eigenvalues which are roots 
of unity of degrees $3,4$ and $6$.
\end{proof}

As an immediate consequence of Theorem~\ref{thm-bound-delta}\eqref{thm-bound-delta-1} 
one has the following generalization of Theorem~\ref{cor-kulikov}:

\begin{theorem}
\label{cor-kulikov-gen}
Let $\cC=\{F=0\}$ be a curve, then the following statements are equivalent:
\begin{enumerate}
 \item\label{cor-kulikov-gen-1}
$\cC$ admits a quasi-toric relation of elliptic type $(3,3,3)$ (resp. $(2,4,4)$, or $(2,3,6)$), 
 \item\label{cor-kulikov-gen-2} 
$\cC$ admits an infinite number of quasi-toric relations of elliptic type 
$(3,3,3)$ (resp. $(2,4,4)$, or $(2,3,6)$). 
\end{enumerate} 
Also,~\eqref{cor-kulikov-gen-1} and~\eqref{cor-kulikov-gen-2} above imply 
\begin{enumerate} 
\setcounter{enumi}{2} 
 \item\label{cor-kulikov-gen-3} 
$\cC$ is a $\delta$-partial curve ($\varphi_\delta(t) \mid \Delta_{\cC,\e}(t)$) 
for $\delta=3$ (resp. $4$, or $6$). 
\end{enumerate} 
 
Moreover, if the singularities of $\cC$ are as in Theorem~\ref{thm-main-alb},
then~\eqref{cor-kulikov-gen-1}, \eqref{cor-kulikov-gen-2}, and~\eqref{cor-kulikov-gen-3} are equivalent. 
\end{theorem}

\begin{proof}
First of all note that this result generalizes Theorem~\ref{cor-kulikov} since 
curves with only nodes and cusps as singularities and non-trivial Alexander polynomial
automatically satisfy that $\Delta_{\cC,\e}(\omega_6)= 0$ for a primitive 6th-root
of unity, and hence they are $\delta$-partial curves (in fact, they are $\delta$-total).

Statement~\eqref{cor-kulikov-gen-1} $\Leftrightarrow$ \eqref{cor-kulikov-gen-2} 
follows from the group structure of quasi-toric relations as exhibited in
Theorem~\ref{thm-bound-delta}\eqref{thm-bound-delta-1}.
Also~\eqref{cor-kulikov-gen-1} $\Rightarrow$ \eqref{cor-kulikov-gen-3} is immediate 
since a quasi-toric relation of $F$ induces an equivariant map 
$V_{\delta}\to E_\delta$, which induces a map $\PP^2\setminus \cC \to \PP^1_{\bar m}$,
where $\bar m=(3,3,3)$, $(2,4,4)$, resp. $(2,3,6)$ according to $\delta=3,4$, resp. $6$. 
This implies that $\Delta_{\cC,\e}(t)\neq 0$, where $\e$ is defined as in~\eqref{eq-e-qt}. 
 
Finally, \eqref{cor-kulikov-gen-3} $\Rightarrow$ \eqref{cor-kulikov-gen-1} under the conditions
of Theorem~\ref{thm-main-alb} is an immediate consequence of Theorem~\ref{thm-main-delta},
since~\eqref{cor-kulikov-gen-3} implies $\rk \MW(W_F)>0$. 
\end{proof}
 
\begin{remark} 
\label{rem-6l} 
Note that~\eqref{cor-kulikov-gen-3} $\Rightarrow$ \eqref{cor-kulikov-gen-1} in
Theorem~\ref{cor-kulikov-gen} is not true for general singularities as the following example
shows. Consider $\cC$ the curve given by the product of 6 concurrent lines. Generically, $\cC$
does not satisfy a quasi-toric relation of type $(2,3,6)$ or $(3,3,3)$ 
(cf. \ref{sixlines}), but its Alexander
polynomial is not trivial, namely, $\Delta_\cC(t)=(t-1)^5(t+1)^4(t^2+t+1)^4(t^2-t+1)^4$. 
\end{remark} 
 
\subsection{Applications}

The Theorem~\ref{thm-main-delta} has implications for the structure of
the characteristic variety $\Sigma(\cC)$ of a plane curve $\cC$. Characteristic 
varieties of curves extend the notion of Alexander polynomials of curves to 
non-irreducible curves (for a definition see~\cite{charvar}). 
In~\cite[Theorem 1.6]{arapura} (resp.~\cite[Theorem 1]{ACM-quasi-projective}), 
structure theorems for the irreducible components of $\Sigma(\cC)$ are given in 
terms of the existence of maps from $\PP^2\setminus \cC$ onto Riemann surfaces 
(resp. orbifold surfaces). The following result 
sharpens~\cite[Theorem 1]{ACM-quasi-projective} for the special case of torsion 
points of order $\delta=3,4,6$ on~$\Sigma(\cC)$.

\begin{cor}
Let $\cC$ be a curve whose singularities are as in Theorem~\ref{thm-main-alb},
$X=\PP^2\setminus \cC$, and consider $\Sigma_1(X)$ the 
first characteristic variety of $X$. For any $\rho\in \Sigma_1(X)$ torsion point 
$\rho$ of order $\delta$ of $\Sigma_1(X)$ there exists an admissible orbifold 
map $f:\PP^2\to S_{\bar m}$ such that $\rho\in f^*\Sigma_1(S_{\bar m})$.
\end{cor}

\begin{proof}
Let $\rho=(\omega_{\delta}^{\e_1},\dots,\omega_{\delta}^{\e_r})\in \Sigma_1(X)$ be a
torsion point of order $\delta$ in $\Sigma_1(X)$. Note that the homomorphism 
induced by $\e:=(\e_1,\dots,\e_r)$ is such that the cyclotomic polynomial 
$\varphi_\delta(t)$ of the $\delta$-roots of unity divides $\Delta_{\cC_e,\e}(t)$. 
Hence the hypotheses of Theorem~\ref{cor-kulikov-gen}\eqref{cor-kulikov-gen-3} are 
satisfied and therefore there exists an orbifold Riemann surface $S_{\bar m}$ such that
$\varphi_\delta(t)$ divides $\Delta_{\pi^{\text{orb}}_1(S_{\bar m})}(t)$ and a 
dominant orbifold morphism $f:X\to S_{\bar m}$. After performing a Stein factorization 
we may assume that the induced homomorphism $f:X\to S_{\bar m}$ is surjective. Finally 
note that any such a surjection induces an inclusion of characteristic varieties, that 
is, $f^*\Sigma_1(S_{\bar m})\subset \Sigma_1(X)$. Moreover, since 
$\varphi_\delta(t)$ divides $\Delta_{\pi^{\text{orb}}_1(S_{\bar m})}(t)$, then 
$\rho\in f^*\Sigma_1(S_{\bar m})$.
\end{proof}

\begin{remark}
Theorem~\ref{cor-kulikov-gen} cannot be generalized to $\delta$-essential
curves outside of the range $\delta=3,4,6$. For instance, 
note that in~\cite{ACM-multiple-fibers} an example of an irreducible affine quintic 
$\cC_5\subset \CC^2$ with $2\AAA_4$ as affine singularities is shown. The curve 
$\cC_5$ is 10-total since $\Delta_{\cC_5}(t)=t^5-t^4+t^3-t^2+t-1$.
However, there is no quasi-toric decomposition of $C_5$ of type $(2,5)$. 
This seems to contradict~\cite[Theorem 1]{kulikov-albanese}.
\end{remark}

\begin{remark}
\label{rem-okas-conj} 
Since no irreducible sextic has an Alexander polynomial of the form $\Delta_{\cC}(t)=(t+1)^q,$ 
(cf.~\cite{degtyarev-alexander}), Theorem~\ref{cor-kulikov-gen} is another way to show that 
any irreducible sextic (with simple singularities) has a non-trivial
Alexander polynomial if and only if it is of torus type (note that an irreducible sextic admits a 
quasi-toric decomposition if and only if it is of torus type as shown in Example~\ref{ex-236}). 
This is a weaker version of what is now known as \emph{Oka's conjecture} 
(cf.~\cite{degtyarev-okas-conjecture-i,degtyarev-okas-conjecture-ii,Oka-Eyral-fundamental}). 
\end{remark} 
 
\begin{remark}\label{sixlines} 
Consider the case when $\cC$ is union of $6$ concurrent lines. To these $6$ lines $\ell_i=0$ 
correspond $6$ points $L_i$ in $\PP^1$ parametrizing lines in the pencil containing $\ell_i$. 
The 3-fold cover $V_{\ell_i}: z^3-\Pi \ell_i$ has as compactification surface in the weighted
projective space $\PP(2,1,1,1)$, which after a weighted blow-up has a map onto the 3-fold cyclic
cover $C_{L_i}$ of $\PP^1$ branched over the points $L_i$ and having rational curves are fibers.
Hence $\Alb(V_{\ell_i})$ is isomorphic to the Jacobian of $C_{L_i}$. If $L_i$ form a collection 
of roots of a polynomial over $\QQ$ having as the Galois group over the later the symmetric
(or alternating) group then the results of~\cite{Zarhin} show that the above
Jacobian is simple and hence does not 
have maps onto the elliptic curve $E_0$. 
 On the other hand the Jacobian of the 3-fold cover of
$\PP^1$ branched over the roots of the polynomial $z^3=x^6-1$ has $E_0$ as a factor, since $E_0$
is a quotient of the former by the involution. Note that the characteristic polynomial of the deck
transformation on $H_1(V_{\ell_i})$ is $(t^2+t+1)^4(t-1)^5$. 
\end{remark} 
 
\begin{remark} 
Note that Oka's conjecture cannot be generalized to general reducible sextics as shown in
Remark~\ref{rem-6l}. 
\end{remark} 
 
However, in light of Theorem~\ref{cor-kulikov-gen} it seems that one can ask the following version
of Oka's conjecture to non-reduced curves, namely: 
 
\begin{question} 
\label{conj-oka-gen}
Is it true that a (possibly non-reduced) sextic $\cC=\{F=0\}$ (with $\deg F=6$) whose singularities
are as in Theorem~\ref{thm-main-alb} has a non-trivial Alexander polynomial if and only if it admits a
quasi-toric relation? 
\end{question} 
 
According to Theorem~\ref{cor-kulikov-gen}, one has the following rewriting of Question~\ref{conj-oka-gen}. 
 
\begin{prop} 
The answer to Question~\ref{conj-oka-gen} is affirmative if and only if 
$\Delta_{\cC,\e}(t)\neq (t-1)^{r-1}(t+1)^q$ for any non-reduced curve $\cC$ 
of total degree $6$ whose singularities are as in Theorem~\ref{thm-main-delta}. 
\end{prop} 
 
\section{Examples}\label{sec-examples} 
The purpose of this section is to exhibit the different examples of elliptic quasi-toric relations. 
In Example~\ref{ex-6-9} (resp.~\ref{ex-4-3}) we present quasi-toric relations of type $(2,3,6)$ and
describe generators for the Mordell-Weil group of elliptic sections, which is of rank 3 (resp.~2). 
In Example~\ref{ex-12-39} we present a cuspidal curve whose Alexander polynomial has the largest
degree known to our knowledge. Finally, Examples~\ref{ex-4-4-2} (resp.~\ref{ex-3-3-3}) show
examples of quasi-toric relations of type $(2,4,4)$ (resp.~$(3,3,3)$) and Example~\ref{ex-2a2-3a5}
shows an example of a curve with both $(2,3,6)$ and $(3,3,3)$ quasi-toric relations. 
 
\begin{example} 
\label{ex-6-9} 
Consider the sextic curve $\cC_{6,9}$ with 9 cusps. One easy way to obtain equations is 
as the preimage of the conic $\cC_2:=\{x^2+y^2+z^2-2(xy+xz+yz)=0\}$ by the Kummer abelian cover
$[x:y:z]\mapsto [x^3:y^3:z^3]$. It is well known that the Alexander polynomial of $\cC_{6,9}$
is $\Delta_{\cC_{6,9}}(t)=(t^2-t+1)^3$. Note that $\cC_2$ belongs to the following pencils:
$$
\begin{matrix}
C_2=(x+y-z)^2+4xz,\\
C_2=(x+z-y)^2+4yx,\\
C_2=(y+z+x)^2-4yz.\\
\end{matrix}
$$
Therefore $\cC_6$ has the following 3 quasi-toric relations:
$$
\begin{matrix}
\gs_0 &\equiv & C_{6,9}=(x^3+y^3-z^3)^2+4(xz)^3,\\
\gs_1 &\equiv & C_{6,9}=(x^3+z^3-y^3)^2+4(yx)^3,\\
\gs_2 &\equiv & C_{6,9}=(y^3+z^3+x^3)^2-4(yz)^3.\\
\end{matrix}
$$
In addition, note that
$$
\begin{matrix}
\gs_3 \equiv \quad
C_{6,9}=
4 (x^2+y^2+z^2+xz+yz+xy)^3- \\
 -3(x^3+y^3+z^3+2(xz^2+x^2z+yz^2+xyz+x^2y+y^2z+xy^2))^2
\end{matrix}
$$
leads to another quasi-toric relation. 
\begin{figure}[h]
\begin{center}
\includegraphics[scale=.3]{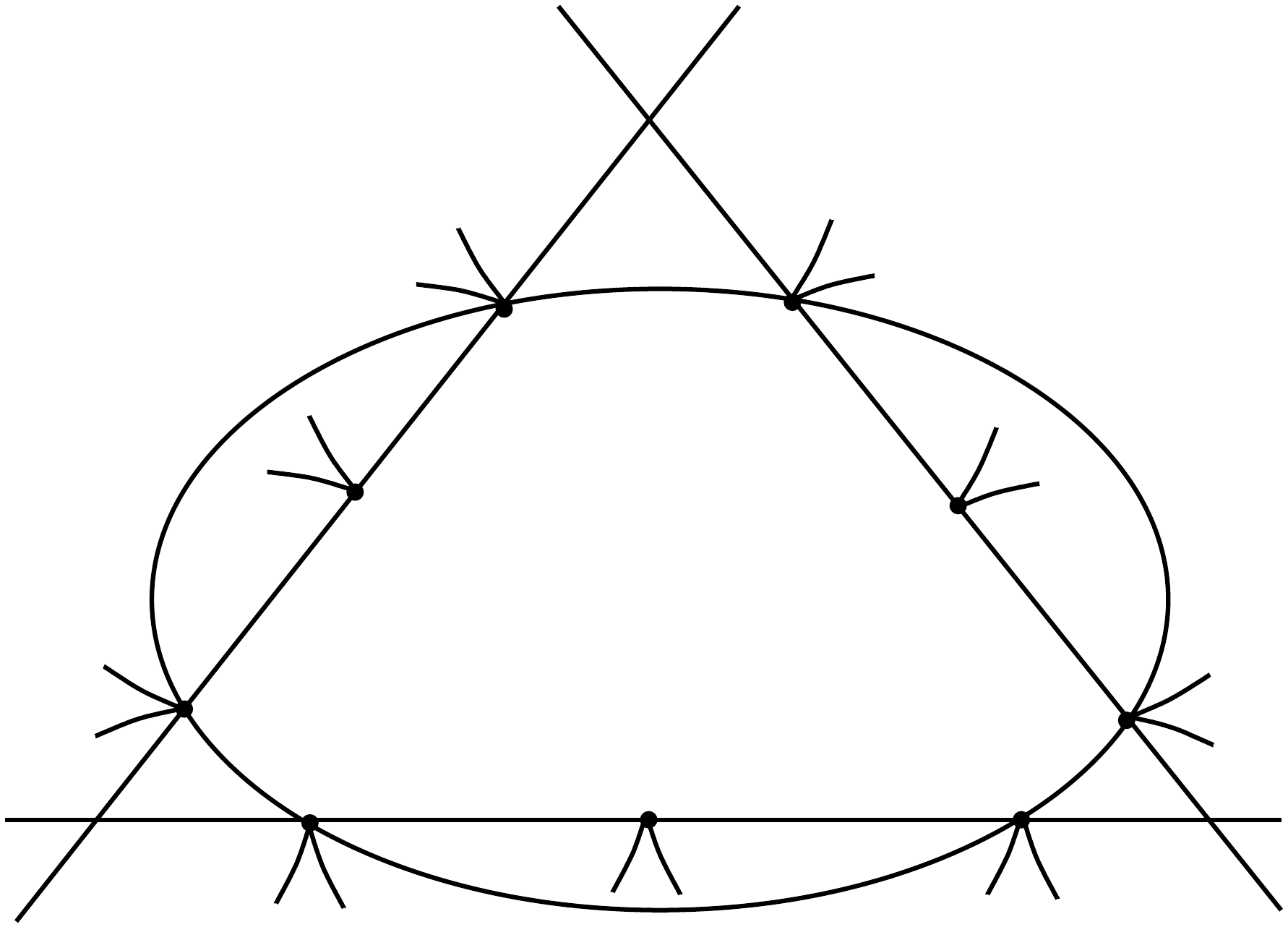}
\end{center}
\end{figure} 
 
If we consider $\gs_i=(g_{2,i},g_{3,i})$ as elements of the elliptic curve $E_0=\{u^3+v^2=F(x,y)\}$ over 
$\CC(x,y)$, then $\ZZ[\omega_6] \gs_1 \oplus \ZZ[\omega_6] \gs_2 \oplus \ZZ[\omega_6] \gs_3$ is the group
of all quasi-toric relations of $\cC_{6,9}$. For instance,
\begin{equation}
\label{eq-rel-c69}
\gs_0=-\gs_1-\gs_2+(2\omega_6-1) \gs_3=\gs_{[1,1,1-2\omega_6]}.
\end{equation}
One can obtain the above relation using~\eqref{eq-action},~\eqref{eq-double}, 
and~\eqref{eq-sum}.

This example was first considered in this context by 
Tokunaga~\cite[Theorem 0.2]{tokunaga-albanese}. 
Note that the author exhibits 12 primitive quasi-toric relations of $\cC_{6,9}$ such that $h=1$. 
Whether or not those decompositions are the only ones satisfying $h=1$ remains open.
\end{example}

\begin{example} 
\label{ex-4-3} 
Consider the tricuspidal quartic $\cC_{4,3}:=\{C_{4,3}=x^2y^2+y^2z^2+z^2x^2-2xyz(x+y+z)=0\}$. 
Since $\cC_{4,3}$ is the dual of a nodal cubic, it should contain a bitangent, which is 
the dual of the node. In our case one can check that $\ell_0:=\{L_0=x+y+z=0\}$
is the bitangent at $P:=[1:\omega_3:\omega_3^2]$ and $Q:=[1:\omega_3^2:\omega_3]$, where 
$\omega_3$ is a root of $t^2+t+1$.

Note that $C_{4,3}L_0^2$ is a non-reduced sextic whose singularities
are $6$-essential (see Remark~\ref{rem-6ess}). Its twisted 
Alexander polynomial w.r.t. the multiplicities $(1,2)$ is given by $(t^3-t+1)^2$. 

By Theorem~\ref{thm-main-delta}, the group of quasi-toric relations has $\ZZ[\omega_6]$-rank one.
In fact, it is generated by the following:
$$
\begin{matrix}
\gs_1 &\equiv & 
C_{4,3}L_0^2=4 C_2^3 + C_3^2\\
\gs_2 &\equiv & 
C_{4,3}L_0^2=4 \tilde C_2^3 + \tilde C_3^2,
\end{matrix}
$$
where
$C_2:=zx+\omega_3 yz-(1+\omega_3)xy$, $C_3:=(x^2y-x^2z-y^2x-3(1+2\omega_3)xyz+y^2z+z^2x-yz^2)$,
$\tilde C_2(x,y,z):=C_2(x,z,y)$, and $\tilde C_3(x,y,z):=C_3(x,z,y)$.

In addition we show some interesting primitive quasi-toric relations coming from 
the combinations of the generators 
$\gs_{[\omega_6^i,\omega_6^j]}:=\omega_6^i \gs_1 + \omega_6^j \gs_2$. 
Such primitive quasi-toric relations have the following form:
$$
\array{lllll}
\gs_{[1,1]} &\equiv & 
C_{4,3}L_0^2x^6& =& 4 C_{4,[1,1]}^3 + C_{6,[1,1]}^2\\
\gs_{[1,\omega_6]} &\equiv & 
C_{4,3}L_0^2(x-z)^6& =& 4 C_{4,[1,\omega_6]}^3 + C_{6,[1,\omega_6]}^2\\
\gs_{[1,\omega_6^2]} &\equiv & 
C_{4,3}L_0^2z^6& =& 4 C_{4,[1,\omega_6^2]}^3 + C_{6,[1,\omega_6^2]}^2\\
\gs_{[1,\omega_6^3]} &\equiv & 
C_{4,3}L_0^2(z-y)^6& =& 4 C_{4,[1,\omega_6^3]}^3 + C_{6,[1,\omega_6^3]}^2\\
\gs_{[1,\omega_6^4]} &\equiv & 
C_{4,3}L_0^2y^6& =& 4 C_{4,[1,\omega_6^4]}^3 + C_{6,[1,\omega_6^4]}^2\\
\gs_{[1,\omega_6^5]} &\equiv & 
C_{4,3}L_0^2(x-y)^6& =& 4 C_{4,[1,\omega_6^5]}^3 + C_{6,[1,\omega_6^5]}^2,
\endarray
$$
where $C_{k,[\omega_6^i,\omega_6^j]}$ denotes a homogeneous polynomial of degree $k$
corresponding
to the quasi-toric relation $\gs_{[\omega_6^i,\omega_6^j]}$. For simplicity 
we only show the first pair of polynomials, $C_{4,[1,1]}$ and $C_{6,[1,1]}$, respectively.
$$
{\small{
\begin{matrix}
\frac{\omega_6}{3}(x^4+z^2y^2+x^2z^2-2xyz^2-2xy^2z+x^3z-2x^2yz+x^3y+x^2y^2)\\
\\
(x^4+x^3y+x^3z-2x^2z^2+4x^2yz-2x^2y^2+4xy^2z+4xyz^2-2z^2y^2)(2x^2+xy-yz+zx)
\end{matrix}
}}
$$
The equations above can easily be obtained using~\eqref{eq-action}, 
\eqref{eq-double}, and~\eqref{eq-sum}.
\end{example}

\begin{example}
\label{ex-12-39} 
Consider the tricuspidal quartic $\cC_{4,3}$ as above, the bitangent $\ell_0$, and two tangent lines, say 
$\ell_1:=\{-8x+y+z=0\}$ (at $R=[1,4,4]$), and $\ell_2:=\{x-8y+z=0\}$ (at $S=[4,1,4]$).
The Kummer cover $[x:y:z]\mapsto [\ell_0^3:\ell_1^3:\ell_2^3]$ produces a curve $\cC_{12,39}$ of degree 12 with 39 cusps.

We compute the superabundance of $\cC_{12,39}$ as follows. Let $\mathcal J$ be the ideal sheaf supported on the 
39 cusps and such that $\mathcal J_\kappa=\mathfrak m$ the maximal ideal at any $\kappa$ cusp of $\cC_{12,39}$.
The superabundance of $\cC_{12,39}$ is $h^1(\mathcal J(7))$ and it coincides with the multiplicity of the 6-th root of
unity as a root of the Alexander polynomial $\Delta_{\cC_{12,39}}(t)$ of $\cC_{12,39}$.

Note that $\chi(\mathcal O(7))-\chi(\mathcal O/\mathcal J)=\chi(\mathcal J(7))=h^1(\mathcal J(7))-h^0(\mathcal J(7))$.
Since $\chi(\mathcal O(7))= {{7+2}\choose{2}} =36$, $\chi(\mathcal O/\mathcal J)=\# \kappa=39$, and $h^0(\mathcal J(7))=1$ 
one has that $h^1(\mathcal J(7))=4$. Note that $h^0(\mathcal J(7))=1$ since the only curve of degree 7 passing through 
all cusps is $z\tilde f_2$, where $\tilde f_2$ is the preimage by the Kummer map of the conic passing through the three cusps, 
$R$, and $S$.

Therefore, $\Delta_{\cC_{12,39}}(t)=(t^2-t+1)^4$. This is, to our knowledge, the first example of a cuspidal curve whose 
Alexander polynomial has a non-trivial root with multiplicity greater than 3. 
As was discussed already in example \ref{ex-6-9} as a cuspidal curve for which the Alexander polynomial
has factors of multiplicity three one can take the dual curve of a smooth 
cubic (its fundamental group was calculated by
Zariski~\cite{zar}). For other examples cf. ~\cite{cog}.
\end{example}

\begin{example} 
\label{ex-4-4-2} 
Consider the moduli space of sextics with three singular points $P$, $Q$, $R$ 
of types $\AAA_{15}$, $\AAA_3$, and $\AAA_1$ respectively. Such a moduli space has 
been studied in~\cite{Artal-Carmona-ji-Tokunaga-sextics} and it consists of two 
connected components $\cM_1$, $\cM_2$. Both have as representatives reducible 
sextics which are the product of a quartic $\cC_4$ and a smooth conic $\cC_2$ 
intersecting at the point $P$ (of type $\AAA_{15}$) and hence $Q,R\in \cC_4$. 
There is a geometrical difference between sextics in $\cM_1$ and $\cM_2$. For 
one kind of sextics, say $\cC_6^{(1)}\in \cM_1$, the tangent line at $P$ also 
contains $Q$, whereas for the other sextics, say $\cC_6^{(2)}\in \cM_2$, it does 
not. The Alexander polynomial of both kinds is trivial; however, if we consider 
the homomorphism 
$\e=(\e_4,\e_2)=(1,2)$, where $\e_i$ is the 
image of a meridian around $\cC_i$, then 
$$\Delta_{\cC^{(i)}_6,\e}(t)=
\begin{cases}
(t^2+1) & \text{ if } i=1,\\
1 & \text{ if } i=2.\\
\end{cases}
$$
Note that $\cC_6^{(1)}$ is a $4$-total curve, but not all of its singularities 
are $4$-essential, since one can check that $\AAA_{15}$ has a local Alexander polynomial
$\Delta_{\AAA_{15},\e}(t)=(1+t^{12})(1+t^6)(1+t^3)(1-t)$. One can also use
Degtyarev's Divisibility Criterion~\cite{degtyarev-divisibility} to prove that
the factors coming from $\Delta_{\AAA_{15},\e}(t)$ do not contribute.

Therefore, Theorem~\ref{thm-main-delta} can be applied and hence $\cC_6^{(1)}$ has a
quasi-toric relation of elliptic type $(2,4,4)$, whereas $\cC_6^{(2)}$ does not.
In particular, these are the equations for the irreducible components 
of~$\cC_6^{(1)}=\cC_4\cup \cC_2$:
\begin{equation}
\array{c}
C_4:=2xy^3+3x^2y^2+108y^2z^2-x^4,\\
C_2:=3x^2+2xy+108x^2,
\endarray
\end{equation}
which fit in the following quasi-toric relation:
$$
C_2h_1^2+h_2^4-C_4h_3^4=0,
$$
where $h_1:=y$, $h_2:=x$, and $h_3:=1$.
\end{example}

\begin{example} 
\label{ex-3-3-3} 
As an example of a quasi-toric relation of elliptic type $(3,3,3)$ we can present
the classical example $F=(y^3-z^3)(z^3-x^3)(x^3-y^3)$. The classical Alexander 
polynomial of $\cC:=\{F=0\}$ is $\Delta_{\cC}(t)=(t^2+t+1)^2(t-1)^8$ and it is 
readily seen that it is a $3$-total curve, since its singularities (besides the nodes) 
are ordinary triple points, which are $3$-essential singularities. Hence one can 
apply Theorem~\ref{cor-kulikov-gen} and derive that $F$ fits in a quasi-toric 
relation of elliptic type $(3,3,3)$. For instance:
\begin{equation}
\label{eq-333}
x^3(y^3-z^3)+y^3(z^3-x^3)+z^3(x^3-y^3)=0.
\end{equation}
However, according to Theorem~\ref{thm-main-alb} there should exist another relation
independent from~\eqref{eq-333}, namely
\begin{equation}
\label{eq-333-2}
\ell_1^3F_1+\ell_2^3F_2+\ell_3^3F_3=0,
\end{equation}
where
$$
F_i=(y-\omega_3^iz)(z-\omega_3^{i+1}x)(x-\omega_3^{i+2}y), \ \ i=1,2,3,
$$
$\omega_3$ is a third-root of unity, and
$$
\array{l}
\ell_1=(\omega_3-\omega_3^2)x+(\omega_3-\omega_3^2)y+(\omega_3^2-1)z,\\
\ell_2=(\omega_3-\omega_3^2)z+(\omega_3-\omega_3^2)x+(\omega_3^2-1)y,\\
\ell_3=(\omega_3-\omega_3^2)y+(\omega_3-\omega_3^2)z+(\omega_3^2-1)x.
\endarray
$$
\end{example}

\begin{example}
\label{ex-2a2-3a5}
In Remark~\ref{rem-total-partial} we have presented a $6$-total sextic curve $\cC$
which is also $3$-partial. In particular, according to Theorem~\ref{cor-kulikov-gen},
$\cC$ admits quasi-toric relations both of type $(2,3,6)$ and $(3,3,3)$. We will show
this explicitly. Note that $\cC$ is the union of two cuspidal cubics 
(cf.~\cite[Corollary 1.2]{oka-alexander}) $\cC_1:=\{F_1=0\}$ 
and $\cC_2:=\{F_2=0\}$ given by the following equations:
$$
\array{c}
F_1=y^3-z^3+3x^2z+2x^3,\\
F_2=y^3-z^3+3x^2z-2x^3.
\endarray
$$
One can check that the curve $\cC$ satisfies the following relations:
$$
\array{c}
3f_3^2-4f_2^3+F_1F_2h^6=0,\\
4x^3-F_1h^3+F_2h^3=0,
\endarray
$$
where $f_2:=yz+y^2+z^2-x^2$, $f_3:=z^3-x^2z-2yx^2+2yz^2+2y^2z+y^3$, and $h=1$.
\end{example}

\end{document}